\documentclass{article}
\usepackage{amsfonts}
\usepackage{amssymb}
\usepackage{bbm}
\usepackage{comment}
\usepackage{enumitem}
\usepackage{hyperref}
\usepackage{dsfont}
\usepackage{amsthm}
\usepackage{amsmath}
\usepackage{commath}
\usepackage{url}
\usepackage{epsfig}
\usepackage{subfigure}
\usepackage{adjustbox}
\usepackage{mathrsfs}  
\usepackage[ruled,vlined,linesnumbered]{algorithm2e}
\usepackage{graphicx}
\usepackage{a4wide}
\usepackage{mathtools}

\makeatletter
\@namedef{subjclassname@2010}{%
	\textup{2010} Mathematics Subject Classification}
\makeatother

\usepackage[T1]{fontenc}
 
\DeclareMathOperator{\supp}{supp}

\DeclareMathOperator{\N}{\mathbb{N}}

\DeclareMathOperator{\bS}{\mathbb{S}}
\newtheorem{thm}{Theorem}
\newtheorem*{que}{Question}
\newtheorem*{thm*}{Theorem}
\newtheorem{lem}[thm]{Lemma}

\newtheorem{prop}[thm]{Proposition}
\newtheorem{cor}[thm]{Corollary}
\newtheorem{lemma}[thm]{Lemma}

\theoremstyle{definition}

\def \Zz {\mathbb{Z}}

\def \cS {\mathcal S}

\def \bl {\mathbf{\ell}}

\def \bl {\mathbb{L}}
\newcommand{\cA}{\mathcal{A}}

\newcommand{\cM}{\mathcal{M}}

\newcommand{\pr}{\mathbb{P}}
\newcommand{\p}{\mathcal{P}}
\newcommand{\A}{\mathcal{A}}

\usepackage{graphicx} % Required for inserting images

\title{The Prime times of twisted Diophantine approximation}
\author{Manuel Hauke}
\date{March 2025}

\begin{document}

\maketitle

\begin{abstract}
    The seminal work of Kurzweil (1955) provides for any fixed badly approximable $\alpha$ and monotonically decreasing $\psi$ a Khintchine-type statement on the set of the inhomogeneous real parameters $\gamma$ for which $\lVert n \alpha + \gamma\rVert \leq \psi(n)$ has infinitely many integer solutions, and further shows that the assumption of $\alpha$ being badly approximable is necessary. In this article, we generalize Kurzweil's statement to restricting $n \in \A $, where $\A \subseteq \mathbb{N}$ is a set with some multiplicative structure. We show that for badly approximable $\alpha$, the result of Kurzweil extends to a general class of sets $\A$, which allows us to establish the Kurzweil-type result in particular along the primes and along the sums of two squares. Furthermore, we construct non-trivial sets $\A$ where the assumption of $\alpha$ being badly approximable is necessary. In particular, this criterion applies to $\A$ being the set of square-free numbers, providing a novel characterization of the badly approximable numbers.\\
    These statements in particular allow for improving the best known bounds for $\lVert n \alpha + \gamma\rVert \leq \psi(n)$ for infinitely many $n \in \A$ for fixed badly approximable $\alpha$ and for various sets $\A$ of number-theoretic interest when accepting an exceptional set for $\gamma$ of Lebesgue measure $0$.
\end{abstract}

\section{Introduction}

One of the main goals in metric Diophantine approximation is to establish so-called Khintchine-type results, named after the famous Khintchine Theorem \cite{Khi_thm}: 

\begin{thm*}[Khintchine's Theorem (1924)]
Let $\psi: \N \to [0,\infty)$ be a monotonically decreasing function.
Writing
\[ W(\psi) := \left\{\alpha \in [0,1): \left\lVert n\alpha \right\rVert \leq \psi(n) \text{ for infinitely many } n \in \N \right\},
\]
where $\lVert.\rVert$ stands for the distance to the nearest integer, we have that

\begin{equation}\label{khin_condition}
\lambda(W(\psi))= \begin{cases}
0 &\text { if }  \sum\limits_{n \in \N} \psi(n) < \infty,\\
    1 &\text{ if } \sum\limits_{n \in \N}\psi(n) = \infty,
\end{cases}
\end{equation}
where $\lambda$ denotes the ($1$-dimensional) Lebesgue measure.
\end{thm*}
This result was generalized by Sz\"usz \cite{Szusz} to the inhomogeneous setup, establishing \eqref{khin_condition} for $W(\psi,\gamma)$ where 
\[ W(\psi,\gamma) := \left\{\alpha \in [0,1): \left\lVert n\alpha + \gamma \right\rVert \leq \psi(n) \text{ for infinitely many } n \in \N \right\}.
\]
While the theorems of Khintchine and Szüsz fix the inhomogeneous parameter $\gamma$ and take $\alpha$ drawn uniformly at random, 
the setup of ``twisted Diophantine approximation'' has recently gained a lot of interest in various aspects (see e.g. \cite{mosh1, mosh2, bdgw, Bsv_mani, chow_zhou, fuchs, mosh3, mum1, kim, kim2, mosh4, mosh5, simmons} for works in twisted Diophantine approximation in the last decade) :
Here, one fixes the parameter $\alpha$, and lets $\gamma$ be drawn uniformly at random. More formally, given a monotonically decreasing function $\psi: \N \to [0,\infty)$ and fixed $\alpha \in [0,1)$, we define
\[T(\psi,\alpha) := \{\gamma \in [0,1): \lVert n\alpha + \gamma \rVert \leq \psi(n) \text{ for infinitely many } n \in \N\}.\]
This can be translated into a question of ``shrinking targets'' (see e.g., the foundational work of Hill and Velani \cite{hv}) in dynamical systems, i.e., we ask about
the visits of orbits $(R_{\alpha}^n(\gamma))_{n \in \N}$ to a shrinking neighborhood around $\alpha$ for generic starting points $\gamma$, where $R_{\alpha}$ denotes the irrational rotation $x \mapsto x + \alpha \pmod 1$. 
In this language, the set $T(\psi,\alpha)$ above describes exactly those $\gamma$ where 

\[\sum_{n \in \N}\mathds{1}_{[R_{\alpha}^n(\gamma) \in B(0,\psi(n))]} = \infty,\]
where $B(x,\rho)$ denotes the ball with center $x$ and radius $\rho$. In this sense, the twisted setup is from a dynamical perspective more natural than the Khintchine setup, since here the dynamical system $R_{\alpha}$ is fixed and the averaging happens only for the starting point, whereas in the Khintchine setup, the averaging happens over the dynamical systems.\\

A naturally arising question is to study an analogue of \eqref{khin_condition}, i.e., to classify pairs $(\psi,\alpha)$ such that

\begin{equation}\label{twisted_BC}
\lambda(T(\psi,\alpha))= \begin{cases}
0 &\text { if }  \sum\limits_{q \in \N} \psi(q) < \infty,\\
    1 &\text{ if } \sum\limits_{q \in \N}\psi(q) = \infty.
\end{cases}
\end{equation}
 As in \eqref{khin_condition}, the convergence Borel--Cantelli-Lemma implies immediately that for all $\alpha \in \mathbb{R}$ and all $\psi$ with $\sum_{n \in \mathbb{N}}\psi(n) < \infty$, \eqref{twisted_BC} holds true. Thus, the question of \eqref{twisted_BC} can be reduced to the set of functions 
 
\[\cM := \left\{\psi: \N \to [0,\infty) \text{ mon. decr., } \sum_{n \in \N} \psi(n) = \infty\right\}\]
and the sets
\[\mathcal{K}(\psi) := \{\alpha \in [0,1): \lambda(T(\psi,\alpha)) = 1\}.\]
The seminal paper by Kurzweil \cite{kurzweil} shows in a beautiful way that the question whether \eqref{twisted_BC} holds depends delicately on the Diophantine properties of $\alpha$:

\begin{thm*}[Kurzweil's Theorem (1955)]
\item
\begin{itemize}
    \item[(i)] For any $\psi \in \cM$ we have $\lambda(\mathcal{K}(\psi)) = 1$. 
    \item[(ii)] \[\bigcap_{\psi \in \cM}\mathcal{K}(\psi) = BAD,\]
    where
    $BAD := \{\alpha \in \mathbb{R}: \liminf_{n \to \infty} n \lVert n \alpha\rVert > 0\}$
    denotes the set of badly approximable numbers.
    \item[(iii)] In particular,
$\bigcap_{\psi \in \cM}\mathcal{K}(\psi)$ has Lebesgue measure zero but Hausdorff dimension $1$.
\end{itemize}
\end{thm*}

In the present article, we examine this question by restricting $n$ to be in certain subsets $\A \subseteq \mathbb{N}$.
We define
\[W_{\A}(\psi,\gamma) := \left\{\alpha \in [0,1): \left\lVert n\alpha + \gamma \right\rVert \leq \psi(n) \text{ for infinitely many } n \in \A \right\},
\]
where $\psi$ is a function in

\[\cM_{\cA} := \left\{\psi: \N \to [0,\infty) \text{ mon. decr.}, \sum_{n \in \cA} \psi(n) = \infty\right\}.\]

Note that in the classical Khintchine setup, this question has been extensively studied by various illustrious researchers, providing analogues to \eqref{khin_condition} along certain sets $\mathcal{A}$, i.e., proving in various situations that for a fixed $\gamma$, we have

\begin{equation}\label{khin_sub}
\psi \in \mathcal{M}_{\A} \implies \lambda(W_{\A}(\psi,\gamma)) = 1.\end{equation}
For example, the Duffin--Schaeffer Theorem \cite{duff} (not to be confused with the Koukoulopuolos--Maynard Theorem \cite{km}, formerly known as the Duffin--Schaeffer Conjecture) implies that for sets $\A$
where in an averaged sense over $\A$, $\varphi(n)/n$ is bounded away from zero, \eqref{khin_sub} holds true for $\gamma = 0$. This is in particular satisfied for $\A$ being the set of primes, or for any set $\A$ that has positive lower density in $\N$. Further sequences where \eqref{khin_sub} was established for arbitrary $\gamma$ include polynomials with integer coefficients \cite{schmidt_poly} as well as lacunary sequences \cite[Chapter 3]{harm_book}. Note however, that \eqref{khin_sub} does not hold for an arbitrary infinite set $\A$: While the Counterexample of Duffin and Schaeffer \cite{duff} and the inhomogeneous generalization due to Ram\'irez \cite{ram_counter} already gives an indication about that, an explicit proof of this fact can be found in the recent work \cite{chpr}. We also remark that there is a complicated criterion formulated by Catlin \cite{catlin}, which was confirmed by Koukoulopoulos and Maynard \cite{km}, which provides a theoretical answer to whether or not \eqref{khin_sub} holds for $\gamma = 0$. However, this condition is not a clean Khintchine-type statement, and is in many instances essentially impossible to check.\\

The purpose of the present article is to study the twisted  concept for subsequences of the integers: For a fixed infinite set $\mathcal{A} \subseteq \mathbb{N}$, instead of \eqref{khin_sub}, we ask about necessary and sufficient conditions on $\alpha$ such that

\begin{equation}\label{twist_sub}
\psi \in \mathcal{M}_{\A} \implies \lambda(T_{\A}(\psi,\alpha)) = 1,\end{equation}
where

\[T_{\A}(\psi,\alpha) := \{\gamma \in [0,1): \lVert n\alpha + \gamma \rVert \leq \psi(n) \text{ for infinitely many } n \in \A\}.\]
More precisely, defining
\[\mathcal{K}_{\A}(\psi) := \{\alpha \in [0,1): \lambda(T_{\A}(\psi,\alpha)) = 1\},\]
 we study the question of whether the analogues of (i) - (iii) of Kurzweil's Theorem hold true for various sets $\A$.\\

In dynamical language, \eqref{twist_sub} asks for 
\[\sum_{n \in \A}\mathds{1}_{[R_{\alpha}^n(\cdot) \in B(0,\psi(n))]} = \infty \text{ almost surely},\]
and thus can be interpreted as the shrinking target problem only stopped at certain times.
Such study can be compared to the seminal works of Bourgain, where for arbitrary measure-preserving dynamical systems $(T,\mathcal{B},\mu,\Omega)$, the convergence of 
\[\frac{1}{\#\A \cap [0,N]}\sum_{\substack{n \in \A\\n \leq N}}f(T^n(\cdot))\]
for $f \in L^p, p > 1$ was considered. This was considered in particular for $\A$ being the set of values of an integer polynomial \cite{bour1}, as well as for $\A$ being the prime numbers \cite{bour2}, and has since been generalized in various setups. The case of prime numbers will also take a central role in this article, but is of course of a different flavor: In our case, the target window is shrinking, but the dynamical system is specialized to be the irrational rotation.\\

Coming back to the analogues of (i) - (iii) in Kurzweil's Theorem for various sets $\A$, we start to consider the analogue of (i), i.e., for fixed $\psi \in \mathcal{M}_{\A}$, does \eqref{twist_sub} hold for Lebesgue-a.e. $\alpha$? It turns out that this is true for \textit{all} infinite sets $\A$ by an application of the doubly inhomogeneous Khintchine Theorem due to Cassels \cite{cas_book}, and a standard application of Fubini's Theorem.
In the past, researchers \cite{hr_twist, kp25} tried to generalize this by replacing the Lebesgue measure on the rotation parameter by other probability measures $\mu$, asking for $\mu(\mathcal{K}_{\A}(\psi)) = 1$ for (specific) $\psi \in \mathcal{M}_{\A}$. Furthermore, the special case of $\psi(n) = \frac{1}{n}, \A = \mathbb{P}$ was considered in recent work of the author with Kowalski \cite{HK}.
The main aim of this article is to study analogues of (ii) and (iii) from Kurzweil's Theorem for various sets $\A$, which, to the best of the author's knowledge, has not been studied prior to this manuscript. The class of sets $\A$ where our analysis will be applicable will include prime numbers and sieve-theoretic generalizations thereof, as well as sequences with positive lower density.

\subsection{Kurzweil's Theorem on prime numbers and generalizations}

The examination of the prime rotation $(p\alpha)_{p \in \mathbb{P}}$ has a rich history, dating back to the foundational works of Vinogradov, Rhin, and Vaughan \cite{vino_primes,rhin,vaughan_primes} who showed that for every irrational $\alpha$, the sequence $(p\alpha)_{p \in \mathbb{P}}$
equidistributes in $[0,1)$. Further, it was proven that for every irrational $\alpha$, we have
\begin{equation}\label{prime_approx}\lVert p \alpha + \gamma \rVert \leq p^{-\tau + \varepsilon} \text{ for i.m. }p \in \mathbb{P},\end{equation}
where Vinogradov \cite{vino_primes} proved $\tau = 1/5$, which was then improved to $\tau = 1/4$ by Vaughan \cite{vaughan_primes}. This result was further strengthened by Harman \cite{harm_primes, harm_primesII} and Jia \cite{JiaI,JiaII} with the current record being $\tau = 9/28$. In the case of $\gamma = 0$, this was further improved by 
Heath--Brown and Jia \cite{HB_Jia}, with the recent best bound $\tau = 1/3$ established by Matomäki \cite{mat_primes}.\\

Similar to the primes, Diophantine approximation with denominators in
\[\mathbb{S}_2 := \{n \in \mathbb{N}: \exists k,\ell \in \mathbb{N}_0: n = k^2 + \ell^2\}\]
has been studied. This was first treated by Cook \cite{C72} (for recent advancements in this topic, see \cite{BaiI,BaiII}) who showed that for all irrational $\alpha$,
\begin{equation}\label{s2_approx}\lVert n \alpha\rVert \leq n^{-\tau + \varepsilon} \text{ for i.m. } n \in \bS_2\end{equation}
where $\tau = 1/2$.
From a sieve-theoretic point of view, $\mathbb{P}$ and $\bS_2$ are related objects since the classical Theorem on the sum of two squares characterizes 
\begin{equation}\label{s2_char} n \in \bS_2 \quad \Longleftrightarrow \quad \forall p^{k}\Vert n \text{ such that } p \equiv 3 \pmod 4,\; k \text{ is even}.\end{equation}
Note that $\bS_2$ is actually denser than the primes: Proven by Landau \cite{landau}, we have 
\[\#\{n \leq N: n \in \bS_2\} \sim c\frac{N}{\sqrt{\log N}}\]
with $c = 0.7642\ldots$ denoting the Landau--Ramanujan constant. 
While some classical analytic tools established for the prime numbers (such as convolution identities via the von Mangoldt function) are not available for $\bS_2$, the latter is actually easier to handle from a sieve-theoretic perspective, since \eqref{s2_char} gives rise to a $1/2$-dimensional sieve compared to the $1$-dimensional sieve that is usually employed in the prime case.\\

The sieve perspective is the point of view we take in this article: We treat both sets $\mathbb{P},\mathbb{S}_2$ as special instances of a more general sieve-theoretic setup. The first theorem of this article allows us to make a statement about twisted Diophantine approximation on certain sieve-theoretic sets that include $\mathbb{P}$ and $\bS_2$ as special cases (see Corollaries \ref{prime_cor} and \ref{s2_cor} below).

\begin{thm}\label{thm_prime_like}
Let $\A \subset \N$ be an infinite set of integers so that there exists $\p \subseteq \pr, \varepsilon,\delta,\rho > 0$ with:

\begin{itemize}
 \item[(I)] For all sufficiently large $k$, we have
 \[n \in \A \cap [2^k,\infty] \implies \forall p \in \p \cap [0,2^{\rho \cdot k}]: p \nmid n.\]
    \item[(II)] \[\frac{1}{(\log x)^{\delta}} \gg \frac{\#\{x \leq n < 2x: n \in \A\}}{x} \asymp \prod_{\substack{p \in \p\\p \leq x}}\left(1 - \frac{1}{p}\right).\]
    \item[(III)] For all $\alpha \in BAD$, we have the discrepancy estimate
    \[\sup_{0 \leq a < b \leq 1} \left\vert\#\{n \in \A \cap [0,x]: \{n\alpha\} \in [a,b]\} - \#\{n \in \A \cap [0,x]\} \cdot (b-a)\right\vert \ll_{\alpha} x^{1- \varepsilon}.\]
    Assuming (I), (II), and (III), we have

    \begin{equation}\label{kw_ii_sieve}
        BAD \subseteq \bigcap_{\psi \in \mathcal{M}_{\A}}\mathcal{K}_{\A}(\psi) \subseteq BAD_{\mathcal{P}}
    \end{equation}
    where 
    \[BAD_{\mathcal{P}} := \{
    \alpha: \liminf_{n \to \infty} n \lVert n \alpha\rVert f_{\p}(n) > 0
    \}, \quad \text{with}\quad f_{\p}(n) := \prod_{\substack{p \leq n\\p \in \p}}\left(1 +\frac{1}{p}\right).\]
    In particular, $\bigcap_{\psi \in \mathcal{M}_{\A}}\mathcal{K}_{\A}(\psi,\alpha)$ has Lebesgue measure $0$, but Hausdorff dimension $1$. 
\end{itemize}
\end{thm}

While the above statement is quite general, we illustrate its application to certain sets $\A$ of number-theoretic interest below. Indeed, for $\p = \mathbb{P}$, we obtain the following by proving (I),(II), and (III) for the respective setups (see Section \ref{subsec_prime_s2} for a detailed proof of Corollaries \ref{prime_cor} - \ref{intersec_cor}):

\begin{cor}[Prime case]\label{prime_cor}
    Let $\alpha$ be badly approximable. Then for any monotonically decreasing $\psi: \mathbb{N} \to [0,\infty)$, we have 
    \[\lambda(\{\gamma \in [0,1): \lVert p\alpha + \gamma \rVert \leq \psi(p) \text{ for infinitely many } p \in \mathbb{P}\}) = 
 \begin{cases}
0 &\text { if }  \sum\limits_{p \in \mathbb{P}} \psi(p) < \infty,\\
    1 &\text{ if } \sum\limits_{p \in \mathbb{P}}\psi(p) = \infty.
\end{cases}
    \]
    However, for $\alpha$ satisfying
    \[\liminf_{n \to \infty} n \lVert n \alpha\rVert \log n = 0,\]
    there exists a monotonically decreasing $\psi = \psi_{\alpha}$ with $\sum\limits_{p \in \mathbb{P}}\psi(p) = \infty$, but 
    \[\lambda(\{\gamma \in [0,1): \lVert p\alpha + \gamma \rVert \leq \psi(p) \text{ for infinitely many } p \in \mathbb{P}\}) = 0.\]
\end{cor}
Similarly, with $\p = \{p \in \mathbb{P}: p \equiv 3 \pmod 4\}$, we obtain\footnote{Strictly speaking, $\mathbb{S}_2$ itself does not satisfy (I). We will use a subset of $\mathbb{S}_2$ with positive relative density that satisfies (I), which in turn allows us to conclude the statement for $\mathbb{S}_2$. See Section \ref{subsec_prime_s2} for details.} the following for the sum of two squares:

\begin{cor}[Case of sum of two squares]\label{s2_cor}
    Let $\alpha$ be badly approximable. Then for any  monotonically decreasing $\psi: \mathbb{N} \to [0,\infty)$, we have 
    \[\lambda(\{\gamma \in [0,1): \lVert n\alpha + \gamma \rVert \leq \psi(n) \text{ for infinitely many } n \in \bS_2\}) = 
 \begin{cases}
0 &\text { if }  \sum\limits_{n \in \bS_2} \psi(n) < \infty,\\
    1 &\text{ if } \sum\limits_{n \in \bS_2}\psi(n) = \infty.
\end{cases}
    \]
    However, for $\alpha$ satisfying
    \[\liminf_{n \to \infty} n \lVert n \alpha\rVert \sqrt{\log n} = 0,\]
    there exists a monotonically decreasing $\psi = \psi_{\alpha}$ with $\sum\limits_{n \in \bS_2}\psi(n) = \infty$, but 
    \[\lambda(\{\gamma \in [0,1): \lVert n\alpha + \gamma \rVert \leq \psi(n) \text{ for infinitely many } n \in \bS_2\}) = 0.\]
\end{cor}

Theorem \ref{thm_prime_like} can be applied to further sequences of number-theoretic importance. Recall the \textit{Löschian integers}
\[\mathbb{L} := \{n \in \mathbb{N}: \exists k,\ell \in \mathbb{Z}: n = k^2 + k\ell + \ell^2\},\] which equals the set of all possible norms of Eisenstein integers.
With $\p = \{p \in \mathbb{P}: p \equiv 2 \pmod 3\}$, we obtain the following for $\mathbb{L}$:

\begin{cor}[Case of Löschian integers]\label{loesch_cor}
Let $\alpha$ be badly approximable. Then for any monotonically decreasing $\psi: \mathbb{N} \to [0,\infty)$, we have 
    \[\lambda(\{\gamma \in [0,1): \lVert n\alpha + \gamma \rVert \leq \psi(n) \text{ for infinitely many } n \in \bl\}) = 
 \begin{cases}
0 &\text { if }  \sum\limits_{n \in \bl} \psi(n) < \infty,\\
    1 &\text{ if } \sum\limits_{n \in \bl}\psi(n) = \infty.
\end{cases}
    \]
    However, for $\alpha$ satisfying
    \[\liminf_{n \to \infty} n \lVert n \alpha\rVert \sqrt{\log n} = 0,\]
    there exists a monotonically decreasing $\psi = \psi_{\alpha}$ with $\sum\limits_{n \in \bl}\psi(n) = \infty$, but 
    \[\lambda(\{\gamma \in [0,1): \lVert n\alpha + \gamma \rVert \leq \psi(n) \text{ for infinitely many } n \in \bl\}) = 0.\]
\end{cor}

We can intersect the sets $\bS_2$ and $\mathbb{L}$, which gives rise to a sieve of dimension $3/4$, with yet another application as follows:

\begin{cor}[Case of $\bS_2 \cap \bl$]\label{intersec_cor}
Let $\alpha$ be badly approximable. Then for any monotonically decreasing $\psi: \mathbb{N} \to [0,\infty)$, we have 
    \[\lambda(\{\gamma \in [0,1): \lVert n\alpha + \gamma \rVert \leq \psi(n) \text{ for infinitely many } n \in \bS_2 \cap \bl\}) = 
 \begin{cases}
0 &\text { if }  \sum\limits_{n \in \bS_2 \cap \bl} \psi(n) < \infty,\\
    1 &\text{ if } \sum\limits_{n \in \bS_2 \cap \bl}\psi(n) = \infty.
\end{cases}
    \]
    However, for $\alpha$ satisfying
    \[\liminf_{n \to \infty} n \lVert n \alpha\rVert (\log n)^{3/4} = 0,\]
    there exists a monotonically decreasing $\psi = \psi_{\alpha}$ with $\sum\limits_{n \in \mathbb{S}_2 \cap \bl}\psi(n) = \infty$, but 
    \[\lambda(\{\gamma \in [0,1): \lVert n\alpha + \gamma \rVert \leq \psi(n) \text{ for infinitely many } n \in \bS_2 \cap \bl\}) = 0.\]
\end{cor}

\subsubsection*{Remarks and open questions}
\begin{itemize}
    \item 

To the best of the author's knowledge, the results mentioned in \eqref{prime_approx}, respectively \eqref{s2_approx},
do not improve when assuming $\alpha$ to have certain Diophantine properties, such as $\alpha$ being badly approximable. Although these results rely on deep number-theoretic considerations, we are very far away from the conjectured bounds of $\tau = 1$ in \eqref{prime_approx} as well as \eqref{s2_approx}. However, Corollaries \ref{prime_cor} and \ref{s2_cor} provide for every fixed badly approximable $\alpha$,

\[\lVert p\alpha + \gamma\rVert \leq \frac{1}{p \log \log p} \text{ for i.m. } p \in \mathbb{P},\]
as well as
\[\lVert n\alpha + \gamma\rVert \leq \frac{1}{n \sqrt{\log n}\log \log n} \text{ for i.m. } n \in \bS_2,\]
for almost every $\gamma$. While the method is purely metric and thus does not allow us to say anything non-trivial for a fixed $\gamma$, we believe the vast improvement in the metrical sense to be interesting on its own, expecting applications in related areas.

\item Corollary \ref{prime_cor} trivially implies that 
for $\sigma := \lambda \otimes \mu$, with $\mu$ an arbitrary probability measure with 
$\supp \mu \subseteq BAD$, that for $\psi \in \mathcal{M}_{\A}$
\[\sigma(\{(\alpha,\gamma) \in [0,1)^2: \lVert p \alpha + \gamma \rVert \leq \psi(p) \text{ for i.m. $p \in \mathbb{P}$}\}) = 1.\]
This can be compared with a result in \cite[Theorem 3.1]{kp25}, respectively \cite[Theorem 1]{hr_twist}, where an analogous statement was established for $\mu$ having positive Fourier decay (such measures are known to exist with its support being contained in $BAD$ \cite{Kaufman}), with a restrictive spacing condition on the sequence $\A$ (that does not hold for e.g. the primes), and only proven for functions of the form $\psi(n) = n^{-\theta}, \theta < 1$. Furthermore, the ``almost sure'' statement is necessary in \cite{hr_twist,kp25}, while in Corollary \ref{prime_cor}, the result holds on \textit{any} fiber $\alpha \in BAD$.

\item Theorem \ref{thm_prime_like} is tailor-made for
general sets
\[\A_{\p} = \{n \in \mathbb{N}: \;\forall p \in \p\cap [1,n^{\rho}]: \gcd(n,p) = 1 \}\]
with $\mathcal{P}$ being a set of positive Dirichlet density, and $\rho < 1$. In these cases, one can straightforwardly establish conditions (I) and (II). In some special cases that satisfy a certain parity condition, we are also able to establish (III) by generalizing the approach of Ramar\'e and  Viswanadham \cite{ramare} (see Corollary  \ref{discr} below). It happens that we can find subsets of relative positive density of $\mathbb{P},\bS_2,\bl,\bS_2\cap \bl$, leading to the respective Corollaries. However, Theorem \ref{thm_prime_like} could be applied to more general sets of integers $\A_{\p}$, provided there is a way of proving (III).

\item The inclusion $\bigcap_{\psi \in \mathcal{M}_{\A}}K_{\A}(\psi) \subseteq BAD_{\mathcal{P}}$ arises from a generalization of the ideas of Kurzweil, and does not use the properties (I),(II), and (III) except for the asymptotic density of $\A$ encoded in (II).
In Proposition \ref{exploit_prime_counterex}, we provide an example that shows, in particular, that for $\A = \mathbb{P}$, we obtain
$\bigcap_{\psi \in \mathcal{M}_{\A}}\mathcal{K}_{\A}(\psi) \subsetneq BAD_{\mathbb{P}}$. This is achieved by the construction of an irrational $\alpha$ for which $\liminf_{k \to \infty}\frac{\varphi(q_k)}{q_k} = 0$, where $(q_k)_{k}$ denotes the sequence of convergent denominators of $\alpha$. While this does not 
rule out a clear cut-off in terms of a classical Diophantine condition that exists in the classical Kurzweil Theorem, it suggests that there is none: the question of whether $\alpha \in \bigcap_{\psi \in \mathcal{M}_{\A}}K_{\A}(\psi)$ depends most probably not only on its Diophantine properties, but also on the arithmetic properties of good approximations to $\alpha$.
In particular, we strongly believe that for $\A = \mathbb{P}$ (and similarly for the other sets $\A$ considered above), there exists no function $f = f_{\A}$ such that 
\[\bigcap_{\psi \in \mathcal{M}_{\A}}\mathcal{K}_{\A}(\psi) = BAD_f,\]
where $BAD_f := \{\lim_{n \to \infty} n f(n)\lVert n\alpha\rVert > 0\}$.
In particular, we do \textit{not} believe that the case $\A = \mathbb{P}$ can be used as a characterization of $BAD$, as it is the case for $\A = \mathbb{N}$ in Kurzweil's Theorem (ii), but this remains open. However, there exist non-trivial sets with positive lower density with this property, which will be discussed in the next section.

\item We remark that Fuchs and Kim \cite{fuchs} provided a full characterization of $\mathcal{K}(\psi)$ by carefully examining the original proof of Kurzweil. Since we follow the original proof for the set inclusion $\bigcap_{\psi \in \mathcal{M}_{\A}}\mathcal{K}_{\A}(\psi) \subseteq BAD_{\mathcal{P}}$, one can work out a necessary condition for $\alpha \in \mathcal{K}_{\A}(\psi)$ of a similar shape. However, since we cannot determine the exact set, a full description remains open. Further, we remark that Chaika and Constantine \cite{chaika} showed that if one does not consider the set of all monotonically decreasing functions $\mathcal{M}$, but only the subset $\mathcal{M}'$ that consists of functions where $q \mapsto q\psi(q)$ is monotonically decreasing, then $\bigcap_{ \psi \in \mathcal{M}'}\mathcal{K}(\psi)$ is a much larger set than $BAD$, and is in fact of full Lebesgue measure. A similar phenomenon might be possible for the generalization to the sets $\A$ above, but the question will not be pursued any further in this article.

\item We want to remark that the implications might also be interesting from a set respectively measure-theoretic point of view, which is as in Kurzweil's classical theorem: For \textit{fixed} $\psi \in \mathcal{M}_{\A}$, we always have $\lambda(\mathcal{K}_{\mathcal{A}}(\psi)) = 1$. However, after taking the (uncountable) intersection over $\mathcal{M}_{\A}$ we are left with a non-empty set, but this set that has measure zero. Note that the intersection cannot be reduced to a countable one, since then the measure would still be $1$.
\end{itemize}

\subsection{Kurzweil on sets with positive density}

As mentioned above, sequences such as the primes most probably do not lead to a characterization of $BAD$, in contrast to (ii) in Kurzweil's Theorem. In this section, we will provide criteria on $\A$ such that $BAD = \bigcap_{\psi \in \mathcal{M}_{\A}} \mathcal{K}_{\A}(\psi)$.
Clearly, this follows from Kurzweil's Theorem for sets $\A$ where $\mathbb{N} \setminus \A$ is finite, but we are interested in nontrivial examples. It turns out that for sets of positive asymptotic lower density,
that is, $\liminf_{x \to \infty} \frac{\#\{n \leq x: n \in \A\}}{x} > 0$,
demanding equidistribution along $\A$ is sufficient:

\begin{thm}\label{thm_pos_dens}
Let $\A \subseteq \N$ be an infinite set of integers with positive lower density, and assume for all $\alpha \in BAD$
that $(n\alpha)_{n \in \A}$ is uniformly distributed. Then 

\[\bigcap_{\psi \in \cM_{\A}}\mathcal{K}_{\A}(\psi) = BAD.\]
\end{thm}
Again, the above statement is very general; obviously, at least density of $(n\alpha)_{n \in \A}$ (within $[0,1)$) for all $\alpha \in BAD$ is necessary: Otherwise, we find $\alpha \in BAD$ and an interval $I \subseteq [0,1)$, with $\{n\alpha\} \notin I$ for all $n \in \mathbb{N}$. Consequently, $\alpha \notin \mathcal{K}_{\A}(\psi)$ for any $\psi$ with $\psi(n) \to 0$, and such $\psi \in \mathcal{M}_{A}$ clearly exists.\\

    Note that when $\A$ has positive lower density, then (by an easy application of summation by parts) $\cM_{\A} = \cM_{\N} (= \cM)$. By making use of estimates on exponential sums twisted by multiplicative functions, we apply Theorem \ref{thm_pos_dens} to a class of sets that are multiplicatively structured in the following way:

\begin{cor}\label{cor_pos_dens_theor}
    Let $m \geq 1$ and let $f: \mathbb{N} \to \mu_m \cup \{0\} \subseteq \mathbb{C}$ be a multiplicative function, where
$\mu_m := \{x \in \mathbb{C}: x^m = 1\}$ denotes the set of $m$-th roots of unity.
Further let $a \in \mu_m \cup \{0\}$ be such that the set
\[\mathcal{A}_{a,f} := \left\{n \in \mathbb{N}: f(n) = a\right\}\]
has positive lower density.
Then
\begin{equation}\label{full_kw}\bigcap_{\psi \in \cM}\mathcal{K}_{\A}(\psi) = BAD.\end{equation}
In particular, \eqref{full_kw} holds true for:
    \begin{itemize}
        \item[(a)] The set of squarefree-numbers $\mathcal{A} = \{n \in \mathbb{N}: \mu^2(n) = 1\} = \{n \in \mathbb{N}: p^k \Vert n \implies k = 1\}$.
         \item[(b)] The set of numbers where the number of prime factors, counted without multiplicity, is congruent $a \pmod m$, i.e.
        $\mathcal{A} = \{n \in \mathbb{N}: \omega(n) \equiv a \pmod m\}$.
        \item[(c)] The set of numbers where the number of prime factors, counted with multiplicity, is congruent $a \pmod m$, i.e.
        $\mathcal{A} = \{n \in \mathbb{N}: \Omega(n) \equiv a \pmod m\}$.
        \item[(d)] Sets that contain an arithmetic progression, i.e. if for fixed $m,a \in \mathbb{N}$,
        $\A \supseteq \{bm + a: b \in \mathbb{N}\}$.
    \end{itemize}
    \end{cor}

\subsubsection*{Remarks and open questions}
\begin{itemize}
    \item 
    We remark again that many of the above results allow us to improve upon known results in denominator-restricted Diophantine approximation when the inhomogeneous parameter is randomized: As the special case of square-free numbers shows, we obtain for $\alpha \in BAD$, 
    \[ \lVert n\alpha + \gamma\rVert \leq \psi(n) \text{ for i.m. } n: \mu^2(n) = 1\]
    for almost every $\gamma$, provided $\sum_{n \in \mathbb{N}} \psi(n) = \infty$.
    In particular, this allows for approximations $\psi(n) = \frac{1}{n \log n \log \log n}$.
    This should be compared with the works of Harman \cite{harm_sqf} and Heath--Brown \cite{HB_squarefree} who showed that for any irrational $\alpha$,
    \[ \lVert n\alpha + \gamma \rVert \leq n^{-\tau + \varepsilon} \text{ for i.m. } n: \mu^2(n) = 1\]
    for $\tau = 1/2$, $\gamma$ arbitrary  \cite{harm_sqf} , and $\tau = 2/3, \gamma = 0$ \cite{HB_squarefree}, respectively. Again, we see that considering $\gamma$ metric allows for a much better approximation quality.\\

\item We emphasize that we obtain a Kurzweil-type result only for sets $\A$ that have positive lower density. The assumption of positive density is crucially used in our proof method, especially for showing 
$\bigcap_{\psi \in \cM_{\A}}{\mathcal{K}}_{\A}(\psi) \subseteq BAD.$
Thus, we raise the following question:

\begin{que}
    Does there exist a set $\A$ with asymptotic density $0$ where $\bigcap_{\psi \in \cM_{\A}}\mathcal{K}_{\A}(\psi) = BAD$? 
\end{que}

\item Finally, we remark that in all sets $\A$ considered in Theorems \ref{thm_prime_like} and \ref{thm_pos_dens}, we had
$BAD \subseteq \bigcap_{\psi \in \cM_{\A}}\mathcal{K}_{\A}(\psi)$. Note that these sets $\A$ are all at least as dense as the primes. As discussed, there are sets $\A$ where this inclusion does not hold, even with positive density; however one might wonder whether there are arbitrarily sparse sets $\A$ where we still have $BAD \subseteq \bigcap_{\psi \in \cM_{\A}}\mathcal{K}_{\A}(\psi)$. This is however ruled out: For lacunary sequences (that is, sequences with $\liminf_{n \to \infty} a_{n+1}/a_n > 1$), it is known \cite[Theorem 1.3]{schmidt_twist_lacunary} that the set of $\alpha$ such that $a_n\alpha$ is not dense in $[0,1)$ is winning in the sense of Schmidt \cite{schmidt_game}. Since the same holds for $BAD$, and intersecting at most countably many Schmidt winning sets gives a Schmidt winning set \cite{schmidt_game}, which in turn always has full Hausdorff dimension, we obtain the existence of a set $S \subseteq BAD$ of full Hausdorff dimension such that for $\A$ lacunary,
$\bigcap_{\psi \in \cM_{\A}}\mathcal{K}_{\A}(\psi) \cap S = \emptyset$.
However, this does not rule out the possibility of $\bigcap_{\psi \in \cM_{\A}}\mathcal{K}_{\A}(\psi)$ being non-empty, thus we ask the following question:

\begin{que}
    Does there exist a set $\A$ such that $\bigcap_{\psi \in \cM_{\A}}\mathcal{K}_{\A}(\psi) = \emptyset$? If not, what can be said about nontrivial lower bounds on the Hausdorff dimension of $\bigcap_{\psi \in \cM_{\A}}\mathcal{K}_{\A}(\psi)$?
\end{que}
\end{itemize}

\subsection{A case study: sums of squares}

We conclude the introductory section with an application of the above results to the well-studied sets 
\[\bS_k := \{n \in \mathbb{N}: \exists (n_1,\ldots,n_k)\in \mathbb{N}_0^{k}: n_1^2 + \ldots + n_k^2 = n\}.\]
By Lagrange's four square theorem, we have for all $k \geq 4$ that $\bS_k = \mathbb{N}$, in which case the classical Kurzweil Theorem shows 
\[\bigcap_{\psi \in \cM_{\bS_k}}\mathcal{K}_{\bS_k}(\psi) = BAD.\]
In the case of $k = 3$, Legendre's three square theorem implies that
\[\bS_3 = \mathbb{N} \setminus \{n \in \mathbb{N}: n=4^{a}(8b+7), a,b \in \mathbb{N}_0\}.\]
We see that $\bS_3$ contains the arithmetic progression 
$\{4n+2, n \in \mathbb{N}_0\}$, thus by Corollary \ref{cor_pos_dens_theor}(d), we have that also when $k = 3$,
\[\bigcap_{\psi \in \cM_{\bS_k}}{\mathcal{K}}_{\bS_k}(\psi) = BAD.\]
When $k = 2$, we may apply Corollary \ref{s2_cor}, showing that 
\[BAD \subseteq \bigcap_{\psi \in \cM_{\bS_k}}\mathcal{K}_{\bS_k}(\psi).\]
Unfortunately, the tools employed in this article do not allow us to make any significant statement when $k =1$, and we ask this as the final open question of this article:

\begin{que}
    Characterize $\bigcap_{\psi \in \cM_{\bS_1}}\mathcal{K}_{\bS_1}(\psi)$ as precisely as possible. In particular, is it true that 
    $BAD \subseteq \bigcap_{\psi \in \cM_{\bS_1}}\mathcal{K}_{\bS_1}(\psi)$?
\end{que}

\subsection*{Acknowledgements}
This work was funded in whole, or in part, by the Austrian Science Fund (FWF). The author was supported by FWF project 10.55776/ESP5134624. He would like to thank Andrei Shubin for encouraging him to provide a general framework that significantly improved the quality of this article. Furthermore, he would like to thank Andrei Shubin and Vivian Kuperberg for related discussions on sums of two squares, and Olivier Ramar\'e for an enlightening discussion about his recent work and possible generalizations. Further, the author would like to thank Victor Beresnevich and Emmanuel Kowalski for discussions on twisted Diophantine approximation, respectively, sieve theory. Finally, he thanks Christoph Aistleitner and Olivier Ramar\'e for useful comments on an earlier version of this manuscript.

\section{Some ideas about the proof and plan of the paper}\label{sec_heur}

In this section, we present some of the core ideas that go into the proofs of the main results of this article.
We start to treat the second inclusion in \eqref{kw_ii_sieve}, as well as $\bigcap_{\psi \in \cM_{\A}}\mathcal{K}_{\A}(\psi) \subseteq BAD$ in Theorem \ref{thm_pos_dens}. This follows from a careful generalisation of Kurzweil's original argument; the somewhat crucial point why a Khintchine-type result fails for well-approximable numbers is that certain Diophantine Bohr sets, i.e. sets
$\{n \leq N: \lVert n\alpha\rVert \leq t_N\}$, contain for some parameters $(N,t_N)$ substantially more points than the expected $2 Nt_N$ many -- see Section \ref{counterex_sec} for details on how to come up with an actual counterexample.\\

For the remainder of the section, we focus on the first set inclusion, i.e., $BAD \subseteq \bigcap_{\psi \in \mathcal{M}_{\A}}\mathcal{K}_{\A}(\psi)$, which is the main difficulty.
We will first start with the ideas for Theorem \ref{thm_pos_dens}, since these ideas are used combined with further input in Theorem \ref{thm_prime_like} again.\\

Put in set-theoretic language, we aim to show that for fixed $\alpha \in BAD$ and fixed $\psi \in \mathcal{M}_{\A}$, we have 
\[\lambda(\limsup_{n \to \infty} A_n),\quad \text{where}\quad A_n := [n\alpha - \psi(n),n\alpha + \psi(n)] \pmod 1,\; n \in \A.\] If the set system were truly stochastically independent, we could apply the divergence Borel--Cantelli Lemma and would be done. However, the set system is far from independent, and thus the aim is to estimate the second moments instead: We will bound the ratio \[r_{\A}(N) := \frac{\sum_{n,m \in [0,N]\cap \A}\lambda(A_n \cap A_{m})}{\left(\sum_{n \in [0,N]\cap \A} \lambda(A_m)\right)^2},\quad \text{as } N \to \infty\] as good as possible from above. If we could establish \textit{quasi-independence on average} (QIA), i.e., $r_{\infty}(\A) := \liminf_{N \to \infty} r_{\A}(N) \leq C$, then we could apply the Chung--Erd\H{o}s inequality, and would obtain $\lambda(\limsup_{n \to \infty} A_n) \geq 1/C$. Optimally, we would wish for 
$C = 1$, but this is for various reasons (that are explained later) difficult to achieve. If we knew a zero-one law, i.e. $\lambda(\limsup_{n \to \infty} A_n) \in \{0,1\}$ for any $(\psi,\alpha,\A)$, then this would clearly imply the statement. Such a law is known due to Cassels \cite{cass01} in the Khintchine setup, but can be easily ruled out here: We could simply choose for fixed $\alpha$ the set $\A := \{n \in \mathbb{N}: \{n\alpha\} \in [0,1/2]\}$, and see that in fact 
$\lambda(\limsup_{n \to \infty} A_n) = 1/2$ for $\psi \in \mathcal{M}_{\A}$.

Thus, we need a different approach to push towards full measure. This is established by following the ideas developed by Allen--Ram\'irez \cite{allen_ram} that build on the work of Beresnevich--Dickinsion--Velani \cite{BDV}.
In a nutshell, this method allows for employing a generalization of the Lebesgue density Theorem to go from $r_{\infty}(\A) < \infty$ to $\lambda(\limsup_{n \to \infty} A_n) = 1$, under the additional assumption that the set system $(A_n)_{n \in \A}$ equidistributes in $[0,1]$. This part of the argument is established in Section \ref{QIA_section}. 
This is actually the only point where the assumption of equidistribution from Theorem \ref{thm_pos_dens} is used. After this, we use the trivial bounds 
\[\sum_{n,m \in [0,N]\cap \A}\lambda(A_n \cap A_{m}) \leq \sum_{n,m \leq N}\lambda(A_n \cap A_{m}), \quad \sum_{n \in [0,N]\cap \A} \lambda(A_m) \gg \delta \sum_{n \leq N} \lambda(A_m),\] where $\delta$ denotes the lower density of $\A$. This shows 
$r_{\infty}(\A) \ll \frac{1}{\delta^2} \cdot r_{\infty}(\mathbb{N})$, reducing the question to the classical Kurzweil setup that can be solved rather straightforwardly -- see Section \ref{thm_pos_dens} for details.\\

In order to prove Theorem \ref{thm_prime_like}, we need to work much harder in the number-theoretic aspect. We will present here the ideas for the case where $\A = \mathbb{P}$, which is actually Corollary \ref{prime_cor}; this covers the core ideas while providing the cleanest setup. As explained before, we aim to show $r_{\infty}(\mathbb{P}) < \infty$; 
After reducing to the case where $\psi$ is constant on dyadic blocks (see Section \ref{sec_wlog_psi}), we essentially are left to bound $
\sum_{p,p' \leq N}\lambda(A_p \cap A_{p'})$, and by elementary estimates, we reach
\begin{equation}\label{symmetric_over}\sum_{p,p' \leq N}\lambda(A_p \cap A_{p'}) \ll \sum_{h \leq N} \mathds{1}_{[\lVert h\alpha \rVert \leq \psi(N)]}\#\{p,p' \leq N: p-p' = h\}.\end{equation}

The quantity on the right-hand side above is clearly related to the problem arising from the (asymptotic) Twin Prime Conjecture (or more generally, Hardy--Littlewood $k$-tuple Conjecture), which is widely open. However, upper-bound sieves work smoothly for this setup, showing that 
\begin{equation}\label{upper_sieve_intro}\#\{p,p' \leq N: p-p' = h\}
\leq C \frac{N}{(\log N)^2} \frac{h}{\varphi(h)}\end{equation}
for some absolute constant $C$ (we replaced for convenience the singular series by the function $\frac{h}{\varphi(h)}$). The argument to achieve \eqref{upper_sieve_intro} in the more general form necessary for Theorem \ref{thm_prime_like} is presented in Section \ref{sec_sieve}. It is quite clear that replacing the value of $C$ in \eqref{upper_sieve_intro} with $1$ is, with currently developed tools, impossible, which makes this the part of Theorem \ref{thm_prime_like} where it is the most crucial to have the argument of Allen--Ram\'irez available that allows for losing constants. 
With \eqref{upper_sieve_intro} at hand, we are left to prove that the function $\frac{h}{\varphi(h)}$ is bounded on average over the Bohr set $\{h \leq N: \lVert h\alpha \rVert \leq \psi(N)\}$. Here we may follow ideas established in the author's previous works \cite{bad_rough,HK} that build on the fact that such Bohr sets have an additive structure, and in particular can be embedded efficiently into a rank $2$ arithmetic progression. Averaging in such a structure (see e.g., the pioneering work of Chow \cite{chow}) is then straightforward, and the result follows, provided the number of elements (i.e. $\approx N \psi(N)$) is not too small. It is this part where it comes in crucially that $\alpha$ is assumed to be badly approximable. The argument is presented in Section \ref{sec_avg_mult}. There might be cases where the Bohr set is too small to get such a result: for instance, the Bohr set can consist of exactly one element $h$, and unfortunately, $h/\varphi(h)$ might be unusually large. However, the contribution of such instances is of negligible order, and this can be found in Section \ref{sec_var} where finally $r_{\infty}(\mathbb{P}) < \infty$ is shown.\\

In fact, the problem is more subtle than presented above: One has to compute an asymmetric version of \eqref{symmetric_over}, that is, one needs to control
$\sum_{p \approx 2^k, p' \approx 2^{\ell}}\lambda(A_p \cap A_{p'})$
for $k\leq \ell$ with $k, \ell$ of potentially different sizes.
The problem is that the upper-bound sieve employed can only establish
\[\#\{p \approx 2^k, p' \approx 2^{\ell}: p-p' = h\} \ll \frac{2^k}{\log (2^k)^2}\frac{h}{\varphi(h)},\]
while the expected size for the left-hand side would be $\frac{2^k}{\log (2^k) \log(2^{\ell})}\frac{h}{\varphi(h)}$. If $\ell \ll k$, then this is just another loss in the constant, which we can afford, but if $\ell$ is much bigger than $k$, this approach clearly fails. However, in such a case we argue differently: We have that for $p \approx 2^k$, the set $A_p$ is an interval of length $2\psi(k)$, which can be assumed to be (see Proposition \ref{wlog_psi}) $\gg \frac{1}{2^{2k}}$. If $\ell$ is now much bigger than $k$, then such an interval looks from the point of view of $2^{\ell}$ almost ``of constant size'', and some quantitative equidistribution result suffices. More precisely, an arbitrary power-saving in the discrepancy of $(p\alpha)_{p \in [2^{\ell},2^{\ell+1}]}$ is enough. This is assumption (III) in Theorem \ref{thm_prime_like}, and holds for primes due to the works of Vinogradov \cite{vino_primes} and Vaughan \cite{vaughan_primes}, and was recently generalized by Ramar\'e and Viswanadham \cite{ramare} to the other setups $\bS_2,\mathbb{L},\mathbb{L} \cap \bS_2$ considered in Corollaries \ref{s2_cor} -- \ref{intersec_cor}. The latter is treated in more detail in Section \ref{subsec_prime_s2}.

\section{Prerequisites}

\subsubsection{Assumptions for $\psi$}\label{sec_wlog_psi}

\begin{prop}[Helpful assumptions for $\psi$]\label{wlog_psi}
Let $\A \subseteq \mathbb{N}$ satisfy either assumption (II) or have positive lower density. Then for any $\theta \in \cM_{\A}$, there exists $\psi \in \cM_{\A}$ with $K_{\A}(\psi) \subseteq K_{\A}(\theta)$ such that 
    \item[(i)] \begin{equation}\label{wlog_i}\psi(n) \geq \frac{1}{n^2}\,\quad  \forall n \in \A.\end{equation}
    \item[(ii)] If $\A$ satisfies (II), then for any $\eta > 0$, \begin{equation}\label{wlog_ii(1)}\psi(n) \leq \frac{\eta f_{\p}(n)}{n}.\end{equation}
    If $\A$ is of positive lower density, then for any $\eta > 0$, \begin{equation}\label{wlog_pos}\psi(n) \leq \frac{\eta}{n}.\end{equation}
    \item[(iii)] $\psi$ is constant on dyadic ranges, i.e., there exist $(\psi_k)_k$ such that
    \begin{equation}\label{wlog_iii}\psi_{|[2^k,2^{k+1})} \equiv \psi_k,\quad k \in \N.\end{equation}
\end{prop}

\begin{proof}\item
\begin{itemize}
\item[(i)] Let $\psi(n) := \theta(n) + \frac{1}{n^2}$. Then $\psi(n) \geq \frac{1}{n^2}$ and
$\sum_{n \in \mathbb{N}}\left(\psi(n) - \theta(n)\right) < \infty$. Thus by the convergence Borel--Cantelli Lemma, we have for every $\alpha \in \mathbb{R}$ that $\lambda(T_{\alpha}(\psi) \setminus T_{\alpha}(\theta)) = 0$. Consequently, since by $\psi \leq \theta$, $\mathcal{K}_{\A}(\psi) \subseteq \mathcal{K}_{\A}(\theta)$, we deduce $\mathcal{K}_{\A}(\psi) = \mathcal{K}_{\A}(\theta)$.
\item[(ii)] 
We claim that for any $\theta \in \mathcal{M}_{\A}$, we have 
$\sum_{n \in \A}\min\{\theta(n),\eta \frac{f_{\p}(n)}{n}\} = \infty$, which is
a generalization of the well-known fact that for $\theta \in \mathcal{M}$, we have 
$\sum_{n \geq 1}\min\{\theta(n),\frac{\eta}{n}\} = \infty$, but we provide a short proof for completeness.
First, we see straightforwardly that $\psi(n) := \min\{\theta(n),\eta \frac{f_{\p}(n)}{n}\}$ is a monotonically decreasing function, since $\eta\frac{f_{\p}(n)}{n}$ is monotonically decreasing, and it remains to show 
$\sum_{n \in \mathbb{N}} \psi(n) = \infty$.
We let \[B := \left\{n \in \A: \theta(n) \leq \eta \frac{f_{\p}(n)}{n} \right\},\quad \text{ and } C := \left\{n \in \A: \theta(n) > \eta \frac{f_{\p}(n)}{n}\right\}.\] If the set $C$ is finite, then obviously $\sum_{n \geq 1}\psi(n) - \theta(n) = O(1)$, and the result follows. Thus we may assume that $C$ is infinite, and therefore we get a sequence of $(n_k)_k$
with $n_k > 2n_{k-1}$ where $\theta(n_k) >\eta \frac{f_{\p}(n_k)}{n_k}$.
By monotonicity of $\psi$, we get for all $n_k/2 \leq n \leq n_k$ that
$\psi(n) \geq \eta \frac{f_{\p}(n_k)}{n_k}$, which implies by (II) that

\[\sum_{\substack{n_k/2 \leq n \leq n_k\\n \in \A}} \psi(k) \geq \eta \frac{f_{\p}(n_k)}{n_k} \cdot \#\{n \in [n_k/2,n_k]: n \in \mathcal{A}\}
\gg \eta \prod_{\substack{p \in \mathcal{P}\\p \leq n_k}} \left(1 + \frac{1}{p}\right)\prod_{\substack{p \in \mathcal{P}\\p \leq n_k}} \left(1 - \frac{1}{p}\right) \gg \eta.
\]

Since this holds for infinitely many $k$, we have $\sum_{n \in \A} \psi(n) = \infty$. The case where $\A$ has positive lower density works in the same way.
    \item[(iii)] 
    This argument follows the lines of the proof for the Cauchy condensation test, but we include it for completeness.
    For $2^k \leq n < 2^{k+1}$, we define $\psi(n) := \theta(2^{k+1})$, which gives
\[\sum_{n \in \A} \psi(n) = \sum_{k \geq 1}\theta(2^{k+1})\mu_k
\]
where $\mu_k := \#\A \cap [2^k,2^{k+1})$.
If $\A$ has positive lower density, then 
$S_k := \sum_{\ell \leq k}\mu_k \gg 2^{\ell}$
and hence by summation by parts,
\begin{equation*}\begin{split}\sum_{n \in \A} \psi(n) &\geq \sum_{\ell \geq 1} S_{\ell} \left(\theta(2^{\ell}) - \theta(2^{\ell+1})\right)
\geq \sum_{\ell \geq 1} \left(2^{\ell} \theta(2^{\ell}) - 2^{\ell-1}\theta(2^{\ell})\right) \\&\gg \sum_{\ell \geq 1} 2^{\ell}\theta(2^{\ell})
\geq \sum_{\ell \geq 1} \sum_{2^{\ell} < n \leq 2^{\ell+1}}\theta(n) = \infty,
\end{split}
\end{equation*}

    since $\theta \in \mathcal{M}_{\A}$. If $\A$ satisfies (II), then note that $\mu_k \asymp \mu_{k+1}$, and thus
    \[\sum_{n \in \A} \psi(n) = \sum_{k \geq 1}\theta(2^{k+1})\mu_k \gg \sum_{k \geq 1}\theta(2^{k+1})\mu_{k+1}
    \geq \sum_{\ell \geq 1} \sum_{2^{\ell} \leq n < 2^{\ell+1}}\theta(n) - O(1) = \infty.
    \]
    Since $\psi \leq \theta$, we clearly obtain $\mathcal{K}_{\A}(\psi)\subseteq \mathcal{K}_{\A}(\theta)$, which concludes the statement.
    \end{itemize}
\end{proof}

\subsection{Notations and standard results}

In view of Proposition \ref{wlog_psi}, we may (and will) assume for the rest of this article that $\psi \in \mathcal{M}_{\A}$ satisfies \eqref{wlog_i}--\eqref{wlog_iii}. For shorter notation, we will define some notations that will be fixed for the remainder of the article:

\begin{itemize}
 \item $D_k := \{2^k \leq n < 2^{k+1}: n \in \A\}$.
 \item $d_k := \prod_{\substack{p \in \p\\p \leq 2^k}}\left(1 + \frac{1}{p}\right)$.
    \item $\mu_k = \#D_k$.
    \item $\psi_k = 2\psi(2^k)$.
    \item $A_n := [n\alpha - \psi(n),n\alpha + \psi(n)] \pmod 1$.
\end{itemize}

We further make use of the standard $O$-and $o$-notations as well as Vinogradov notations $\ll,\gg$, meaning
$f \ll g \Leftrightarrow f = O(g)$, with any dependence of the implied constants denoted by a subscript. If $f \ll g$ and $g \ll f$, we write $f \asymp g$, and $f \sim g$ for $\lim_{x \to \infty} \frac{f(x)}{g(x)} = 1$.
A set of integers $\A \subseteq \mathbb{N}$ is said to have positive lower density if 
$\liminf_{N \to \infty} \frac{\#\{n \leq N: n \in \mathcal{A}\}}{N} > 0$.
We define the complex exponential $e(x) := \exp(2 \pi i x)$ and for a prime $p \in \mathbb{P}$, we write $p^k \lVert m$ if $p^k \mid m$, but $p^{k+1} \nmid m$.
We write as usual $\tau(\cdot)$ for the number of divisors, $\omega(\cdot)$ for the number of prime factors without multiplicity, and $\Omega(\cdot)$ for the number of prime factors with multiplicity. 
\\

For a real $\alpha$, we write $\alpha = [a_0;a_1,a_2,\ldots]$ for its continued fraction expansion, and denote by $\frac{p_n}{q_n} := \frac{p_n(\alpha)}{q_n(\alpha)} := [a_0;a_1,a_2,\ldots,a_n]$ with $p_n,q_n$ coprime the corresponding convergents.
We define by $\{,\}$ and $\lVert.\rVert$ the signed resp. unsigned distance to the nearest integer, and recall the following relations that can be found in any standard literature on continued fractions.

\begin{itemize}
    \item $q_n \lVert q_n\alpha\rVert \asymp \frac{1}{a_{n+1}}$.
    \item $q_n,q_{n-1}$ are coprime for all $n \geq 1$.
\end{itemize}

\subsection{Reduction to showing QIA}\label{QIA_section}

As discussed in Section \ref{sec_heur}, the purpose of this Section is to provide a tool that allows us to go from $r_{\infty}(\A) < \infty$ to $\lambda(\limsup_{n \to \infty} A_n) = 1$ - see Lemma \ref{bdv_applied} below. We make use of a result of Beresnevich--Dickinson--Velani \cite{BDV} in the spirit of Allen--Ram\'irez \cite[Proposition 1]{allen_ram} that can be generalized as follows:
\begin{lemma}\label{ramirez_lem}
    Let $(E_i)_{i\geq 1}$ be a sequence of Lebesgue-measurable subsets of $[0,1)$.  Suppose that there exist positive real numbers $C,K,c > 0$ and a sequence $(\cS_r)_{r \in \N}$ of finite subsets of~$\Zz$ such
  that the following conditions hold:
  \begin{gather}
    \label{eqn00}
    \lim_{r\to+\infty}\min\cS_r=+\infty,
    \\
    \label{eqn04}
    \sum_{i\in\cS_r}\lambda(E_i) \ge c,
    \\
    \label{eqn05}
    \text{ for all sufficiently large $r\geq 1$, we have}
    \\
      \sum_{\substack{s<t\\[0.5ex] s,t\in\cS_r}}
    \lambda\Bigl(E_s\cap E_t \Bigr) \leq  C
    \Bigl(\sum_{i\in\cS_k}\lambda(E_i)\Bigr)^2\notag,
    \\
    \label{vb89}
    \text{for any $\delta>0$ and any  interval $I = [a,b] \subset
      [0,1)$, there exists $r_0 = r_0(I)$ such that }
    \\
    \sum_{s \in \cS_r}\lambda\Bigl(I \cap E_s\Bigr)\geq
    \frac{1}{K}\sum_{s \in \cS_r}\lambda\left(I\right)\lambda(E_s)  
    \quad
    \text{ for all $r \geq r_0$}.\notag
  \end{gather}
  
  Then we have
  \[
    \lambda(\limsup_{i \to \infty} E_i)=1.
  \]
\end{lemma}

Lemma \ref{ramirez_lem} might be of independent interest. We remark the similarity to the results in \cite{BC} where similar conditions were examined. Interestingly, in contrast to \eqref{vb89}, the conditions in \cite{BC} state rather the opposite, by assuming $\sum_{s \in \cS_r}\lambda\Bigl(I \cap E_s\Bigr)\leq
    (1 + \varepsilon)\sum_{s \in \cS_r}\lambda\left(I\right)\lambda(E_s)$ instead. We remark that the results in \cite{BC} need $(1 + \varepsilon)$ instead of an arbitrary $K$, but apply to more general measures $\mu$, while Lemma \ref{ramirez_lem} makes use of some generalized form of Lebesgue density (which is inside the proof of Proposition \ref{BDVdensitylemma} stated below), and is thus restricted to a certain class of measures that includes the Lebesgue measure:

\begin{prop}[Beresnevich--Dickinson--Velani,~{\cite[Lemma 6]{BDV}}]\label{BDVdensitylemma}
  Let $(X,d)$ be a metric space with a finite measure $\mu$ such that
  every open set is $\mu$-measurable. Let $A$ be a Borel subset of $X$
  and let $f:(0,\infty) \to (0,\infty)$ be an increasing function with $f(x)\to 0$ as
  $x\to 0$. If for every open set $U\subseteq X$ we have
  \begin{equation*}
    \mu(A\cap U) \geq f(\mu(U)),
  \end{equation*}
  then $\mu(A) = \mu(X)$.
\end{prop}

\begin{proof}[Proof of Lemma \ref{ramirez_lem}]
This follows the strategy of Allen-Ram\'irez \cite{allen_ram}, but in the more generalized setup provided in Lemma \ref{ramirez_lem}. We fix an open set $U \subseteq [0,1)$ and aim to 
    prove 
     \begin{equation}\label{pos_ever}
    \lambda(\limsup_{i \to \infty}E_i\cap U) \geq \frac{C'}{\lambda(U)^2},
  \end{equation}
  where $C'$ is a later specified constant that is independent of $U$.  
   We will then apply Proposition \ref{BDVdensitylemma} with $f(x) := \frac{C'}{x^2}$ to conclude the proof.\\

   Note that
\[\limsup_{i \to \infty} E_i = \bigcap_{i \geq 1}\bigcup_{j \geq i}E_i,\]
thus by continuity of measures, and since $\lambda_{[0,1)}$ is a finite measure, for all $\delta > 0$, there exists $I_0 = I_0(\delta)$ such that
\[\lambda\left(\bigcup_{i \geq I_0}E_i\right) \geq \lambda(\limsup_{i \to \infty} E_i) - \delta.\]
Since $\bigcup_{i \geq I_0}E_i \supseteq \bigcup_{i \in \cS_r}E_i$ for all $r$ sufficiently large, we obtain
\begin{equation}\label{cont_meas}\begin{split}\lambda(\limsup_{i \to \infty} E_i \cap U)
&\geq \lambda\big(\bigcup_{i \geq I_0}E_i \cap U\big) - \delta
\\&\geq \lambda\big(\bigcup_{i \in \cS_r}E_i \cap U\big) - \delta.
\end{split}\end{equation}
We now apply the Chung--Erd\H{o}s inequality, proving for any $r \in \mathbb{N}$ that
  \[\lambda\left(\bigcup_{i \in \cS_r} E_i \cap U \right)
\geq \frac{
\left(\sum_{i \in \cS_r} \lambda(E_i \cap U)\right)^2
}{\sum_{i,j \in \cS_r} \lambda(E_i \cap E_j \cap U)}
\geq \frac{
\left(\sum_{i \in \cS_r} \lambda(E_i \cap U)\right)^2
}{\sum_{i,j \in \cS_r} \lambda(E_i \cap E_j)}.
\]
We then claim that by \eqref{vb89} we have for $r$ sufficiently large that there exists $K' > 0$ such that

\begin{equation}\label{local_dens}\left(\sum_{i \in \cS_r} \lambda(E_i \cap U)\right)^2 \geq \frac{\lambda(U)^2}{K'}\left(\sum_{i \in \cS_r} \lambda(E_i)\right)^2.\end{equation}
Indeed, since every open set $U \subseteq [0,1)$ is a disjoint union of at most countably many open intervals, a collection of finitely many such intervals provides a subset of $U$ with the sum of its measures being at least $\lambda(U)/2$. By changing the constant $K'$ in \eqref{local_dens} by a factor of $4$, this reduces \eqref{local_dens} to the case where $U$ is an interval, which now follows from \eqref{vb89}.\\ 

On the other hand, by \eqref{eqn05}, we have for sufficiently large $r$

\[\frac{
\left(\sum_{i \in \cS_r} \lambda(E_i)\right)^2
}{\sum_{i,j \in \cS_r} \lambda(E_i \cap E_j)} \geq \frac{1}{C}.\]
Thus, combining the above, we obtain

\[\lambda\left(\bigcup_{i \in \cS_r} E_i \cap U \right)
\geq \frac{\lambda(U)^2}{CK'}.
\]
Finally, we set $\delta = \frac{\lambda(U)^2}{2CK'}$, showing by \eqref{cont_meas} that for all $r$ sufficiently large, 
\[\begin{split}\lambda(\limsup_{i \to \infty} E_i \cap U)& \geq \lambda\big(\bigcup_{i \in \cS_r}E_i \cap U\big) - \frac{\lambda(U)^2}{CK'}  \geq \frac{\lambda(U)^2}{CK'}  - \frac{\lambda(U)^2}{2CK'} = \frac{\lambda(U)^2}{2CK'}.
\end{split}
\]
This proves \eqref{pos_ever} with $C' = 2CK'$, which finishes the proof.
\end{proof}

The following statement finally applies Lemma \ref{ramirez_lem} to bring it into the setup we need for the main theorems in this article.

 \begin{lemma}[Application of the above]\label{bdv_applied}
    Assume that $0 \leq \mu_k\psi_k \leq 1$ for all $k$, and let 
    \[\sum_{k \geq 1}\mu_k\psi_k = \infty.\]
    Further assume $\psi(k) \to 0$ as $k \to \infty$, and assume that there exists $C > 0$ such that for any 
    $[X,Y] \subseteq \mathbb{N}$ with $X$ sufficiently large and 
    \[\sum_{X \leq k \leq Y}\mu_k\psi_k \in [1,2],\]
    we have
    \begin{equation}\label{qia}
        \sum_{X \leq k,\ell \leq Y}\sum_{i \in D_k}\sum_{j \in D_{\ell}}\lambda(A_i \cap A_j) \leq C.
    \end{equation}
    Additionally, assume $(n\alpha)_{n \in D_k}$ equidistributes as $k \to \infty$, i.e., for all intervals $I \subseteq [0,1)$, we have
    \begin{equation}\label{equi}
        \frac{\#\{n \in D_k: \{n\alpha\} \in I\}}{\mu_k} \to \lambda(I), \quad k \to \infty.
    \end{equation}
    Then \[\lambda(\limsup_{i \to \infty}A_i) = 1.\]
\end{lemma}

\begin{proof}
    We check that all conditions of Lemma \ref{ramirez_lem} are satisfied.
    We choose blocks $[X_r,Y_r], r \in \mathbb{N}$ pairwise disjoint where 
     \[\sum_{X_r \leq k \leq Y_r}\mu_k\psi_k \in [1,2].\]
     This can be done since $\sum_{k \geq 1}\mu_k\psi_k = \infty$ and $0 \leq \mu_k\psi_k \leq 1$. We now set $\cS_r = \bigcup_{X_r \leq k \leq Y_r}D_k$, which immediately satisfies \eqref{eqn00} and \eqref{eqn04}. By \eqref{qia}, we obtain immediately \eqref{eqn05}, and we are left to show \eqref{vb89}.
     Since $\psi_k \to 0$, defining for $I = [a,b]$ $\tilde{I} = [a + \psi_k, b - \psi_k]$, we may assume that 
     $\lambda(\tilde{I}) \geq \lambda(I)/2$.
     Further, we get
     $\sum_{i \in D_k}\lambda(I \cap A_i)
     \geq \psi_k \#\{i \in D_k: \{i\alpha\} \in \tilde{I}\}
     $.
     By \eqref{equi}, this implies for $r$ sufficiently large that
     \[\sum_{X_r \leq k \leq Y_r}\sum_{i \in D_k}\lambda(I \cap A_i) \geq \frac{\lambda(I)}{4}\sum_{X_r \leq k \leq Y_r}\mu_k\psi_k,\]
     proving \eqref{vb89}. An application of Lemma \ref{ramirez_lem} now proves the claim.
\end{proof}

\section{Proof of Theorem \ref{thm_prime_like}}

We first prove that $BAD \subseteq \bigcap_{\psi \in \mathcal{M}_{\A}}K_{\A}(\psi)$. This is the main part of the proof --- the part where $\bigcap_{\psi \in \mathcal{M}_{\A}}K_{\A}(\psi) \subseteq BAD_{\p}$ will be shown in Subsection \ref{counterex_sec}. We now fix $\alpha \in BAD$ and $\psi \in \mathcal{M}_{\A}$. By Proposition \ref{wlog_psi} we may (and will) assume that $\psi$ satisfies the conditions \eqref{wlog_i} -- \eqref{wlog_iii}.
By (II), we see that $\mu_k \gg \frac{2^k}{k} \gg 2^{k(1 - \varepsilon/2)}$ since $\prod_{\substack{p \in \p\\p \leq x}}\left(1 - \frac{1}{p}\right) \gg \frac{1}{\log x}$. Thus by (III), we have for all intervals $I$ that
\begin{equation*}%\label{equi_}
        \frac{\#\{n \in D_k: \{n\alpha\} \in I\}}{\mu_k} \to \lambda(I), \quad k \to \infty.
    \end{equation*}
    Additionally, assumption \eqref{wlog_ii(1)} and (II) show that $\mu_k\psi_k \leq 1$, since
    \[\mu_k\psi_k \ll \eta \prod_{\substack{p \in \mathcal{P}\\p \leq n_k}} \left(1 + \frac{1}{p}\right)\prod_{\substack{p \in \mathcal{P}\\p \leq n_k}} \left(1 - \frac{1}{p}\right),\] and choosing $\eta$ sufficiently small proves the claim.
    Hence all conditions of Lemma \ref{bdv_applied} are satisfied, except potentially \eqref{qia}. 
    Thus, it remains to show that
    \begin{equation*}%\label{qia_}
        \sum_{X \leq k,\ell \leq Y}\sum_{i \in D_k}\sum_{j \in D_{\ell}}\lambda(A_i \cap A_j) \leq C,
    \end{equation*}
    whenever $X,Y$ are such that $\sum_{X \leq k \leq Y}\mu_k\psi_k \in [1,2]$, which will be the remainder of the proof.

\subsection{Step 1: A sieve estimate for the second-order correlation}\label{sec_sieve}
We present here an application of a standard upper-bound sieve. This generalizes the upper-bound sieve of the classical Twin prime setup, that is (see e.g. \cite[Theorem 6.7]{ant}),
\[\#\{p\leq x,\mid\, p+h \in \mathbb{P}\} | \ll
      \frac{h}{\varphi(h)}\frac{x}{(\log x)^2},\]
       to a more general situation where we only sift out the primes in $\mathcal{P}$.
\begin{prop}\label{sieve_prop}
    Let $\A ,\p, \rho$ satisfy (I) of Theorem \ref{thm_prime_like}.
    Then we have for all $h \geq 1$
\[
      \#\{n \in D_k: n + h \in \A\} \ll_{\rho}
      \frac{h}{\varphi(h)}\frac{2^k}{d_k^2},
    \]
    where $\varphi(h)$ denotes the Euler totient function. The implied constant is uniform in $h$.
    \end{prop}
    \begin{proof}
    Let us write $\p_n := \prod_{\substack{p \in \p\\p \leq n}}p$.
        By (I), we have 
        \[ \#\{n \in D_k: n + h \in \A\} 
        \leq  \#\{n \in [2^{k},2^{k+1}): \gcd(n(n + h),\p_{2^{\rho \cdot k}}) = 1\}. 
        \]
        We may assume without loss of generality that $\rho < 1/100$, since otherwise, we can always reduce the value of $\rho$, which will only potentially increase the upper bound.
        Furthermore, we may assume that if $2 \in \p$, then $h$ is odd; otherwise, the count above is $0$, and the statement holds trivially.\\
        
        We are now in position to apply upper-bound sieves efficiently: Defining for $p \mid \p_{2^{\rho \cdot k}}$
        \[\nu(p) := \begin{cases} 1 &\text{ if } p \mid h,\\
        2 &\text{ if } p \nmid h,
        \end{cases}\]
        and extending $\nu$ multiplicatively, we obtain for $d \mid \p_{2^{\rho \cdot k}}$
        \[\#\{n \in [2^{k},2^{k+1}): d \mid n(n+h) \} = 2^{k}\cdot \frac{\nu(d)}{d} + O(\tau(d)).        
        \]
        Thus we obtain immediately
        \[\sum_{\substack{d \leq 2^{k/2}\\d \mid \p_{2^{\rho \cdot k}}}}\tau(d)^2 \ll 2^{\frac{2}{3}k}.\]
        This allows us to apply the Fundamental Lemma of Sieve theory in the form of \cite[Theorem 18.11(b)]{kouk_book} (with $\kappa = 2$, $D = 2^{k/2}, X = 2^k, y = 2^{k \cdot \rho}$), showing that
        \[\#\{n \in [2^{k},2^{k+1}): \gcd(n(n + h),\p_{2^{\rho \cdot k}}) = 1\} \ll 2^{k}\prod_{\substack{p \in \p_{2^{\rho \cdot k}}\\p > 2}}\left(1 - \frac{\nu(p)}{p}\right).
        \]
        Now observe that since $\prod_{x \leq p \leq x^{1/\rho}}\left(1 + \frac{1}{p}\right) \ll_{\rho} 1$, we have
        \[\prod_{\substack{p \in \p_{2^{\rho \cdot k}}\\p>2}}\left(1 - \frac{\nu(p)}{p}\right) \asymp \prod_{\substack{p \in \p_{2^{\rho \cdot k}}\\p>2}} \left(1 - \frac{2}{p}\right)\prod_{\substack{p \in \p_{2^{\rho \cdot k}}\\p \mid h}} \left(1 + \frac{1}{p}\right)
        \ll_{\rho} \prod_{\substack{p \in \p_{2^{k}}}}\left(1 - \frac{1}{p}\right)^2 \prod_{\substack{p \in \mathbb{P}\\p \mid h}}\left(1 + \frac{1}{p}\right) \ll \frac{1}{d_k^2}\frac{h}{\varphi(h)}.
        \]
        This finishes the proof.
    \end{proof}

\subsection{Step 2: Averaging multiplicative functions over Bohr sets}\label{sec_avg_mult}

Given an integer~$\ell$, a real number $t > 0$, and an irrational $\alpha$, we denote
\[
  \mathcal{N}_{\alpha}(\ell,t)
  := \Bigl\{1 \leq N \leq 2^{\ell}: \lVert N \alpha
  \rVert \leq t\Bigr\},
\]
which are examples of \emph{Bohr sets}. In this part of the proof, we aim to show that the
function $n/\varphi(n)$ is bounded on average over such
sets. Here we will follow the main ideas of \cite{HK} and \cite{bad_rough}. 
In these articles, we had $2^{\ell}t \to \infty$ (which is the size of the number of elements we average over) with a certain speed. While this allowed us to obtain even the correct average asymptotics for $n/\varphi(n)$ in \cite{bad_rough}, in this article we are allowed to lose constants, but we can not assume anything about $2^{\ell}t$.\\

To prove the result,
we will use the relation between Bohr sets and generalized arithmetic
progressions: For integers $x$, $y$, $z$, we denote
\[
  P(x,y,z) := \{ax + by\,\mid\, a,b \in \Zz\text{ and } \lvert
  a\rvert,\lvert b\rvert \leq z\},
\]
which is a (rank~$2$) finite generalized arithmetic progression. This structural result has been exploited already in earlier work of Chow~\cite{chow}, in a related setup. The following is a straightforward generalization of \cite[Lemma 7.6]{HK}, with the ideas originating in Tao's blog post \cite{tao_blog} and the work of Chow~\cite{chow}:

\begin{lemma}[Structure of the Bohr sets]\label{bohr_AP}
 Let $\alpha$ be a badly approximable number, $\ell \in \N, t> 0$ with $2^{\ell}t > 1$. Then there exist integers $(x,y,z)$, with $x$ and~$y$
  coprime and
  \[
    x \asymp y \asymp_{\alpha} \sqrt{2^{\ell}t^{-1}},\quad\quad z \ll_{\alpha} \sqrt{2^{\ell}t},
  \]
  such that
  \begin{equation}
    \label{bohr_inclusion}
    \mathcal{N}_{\alpha}(\ell,t) \subseteq P(x,y,z).
  \end{equation}
  In particular, it follows that
  \begin{equation}\label{cardinality_bohr}
    \#\mathcal{N}_{\alpha}(\ell,t) \ll_{\alpha} 2^{\ell}t.
  \end{equation}
\end{lemma}

\begin{proof}
Without loss of generality, we will assume $t = 2^{-i}$ for some $i \in \mathbb{N}$ (otherwise, for $2^{-(i+1)} \leq t < 2^{-i}$, we replace $t$ with $2^{-i}$).
Following \cite[Lemma 7.6]{HK}, we see that
  \[
    \mathcal{N}(\ell,2^{-i}) \subseteq P(q_{r},q_r + q_{r-1},z)
  \]
  for any $r\geq 1$, where (recall that $\frac{p_r}{q_r}$ denotes the convergents to $\alpha$)
  \[
    z= \max\Bigl\{ 2^{\ell} \lVert q_r \alpha \rVert + 2^{-i}q_r, 2^j
    \lVert (q_r + q_{r-1})\alpha \rVert + 2^{-i}(q_r + q_{r-1})
    \Bigr\}.
  \]
  Since $\alpha$ is badly approximable, we find $q_r$ at any scale, $q_{r-1} \asymp_{\alpha} q_r$ and $\lVert q_r\alpha\rVert \asymp_{\alpha} \lVert q_{r-1}\alpha\rVert \asymp \frac{1}{q_r}$.
  We optimize for $z$ by picking $r$ such that $q_{r-1} \asymp_{\alpha} 2^{\tfrac{\ell -i}{2}}$, which implies  $z \ll \sqrt{2^{\ell}2^{-i}}$.
  Since the denominators of
  consecutive convergents are always coprime, we also have
  $\gcd(q_{r},q_r + q_{r-1}) = 1$, which concludes the proof
  of \eqref{bohr_inclusion}, while \eqref{cardinality_bohr} follows immediately from $\#P(x,y,z) \ll z^2$.
\end{proof}

  \begin{lemma}[Averaging $n/\varphi(n)$ over Bohr sets]\label{bohr_average}
      Let $\alpha$ be a badly approximable number with 
      $c(\alpha) := \sup_{q \geq 1}q \lVert q\alpha\rVert > 0, \ell \in \N, t> 0$.
      Setting $Y = \sqrt{2^{\ell}t^{-1}}, Z = \sqrt{2^{\ell}t}$, we have for any $\xi > 0$,
\begin{equation}\label{cases_bohr}
    \sum_{N \in \mathcal{N}_{\alpha}(\ell,t)} \frac{N}{\varphi(N)} \ll_{\alpha,\xi} \begin{cases}
    0 &\text{ if } \sqrt{c(\alpha)} > Z,\\
     Z^2 \log {\ell} &\text{ if } \sqrt{c(\alpha)} < Z < (\log Y)^{\xi},\\
    Z^2 &\text{ if } (\log Y)^{\xi} \leq Z.
    \end{cases}
\end{equation}
  \end{lemma}

  \begin{proof}
      Since $\max_{1 \leq n \leq 2^{\ell}}n\lVert n \alpha\rVert \geq c(\alpha)$ by definition, $\mathcal{N}_{\alpha}(\ell,t) = \emptyset$ if $\sqrt{c(\alpha)} > Z$. For the remainder, by increasing $t$ if necessary by a constant factor, we can assume now that $Z > 1$. Thus we can apply Lemma \ref{bohr_AP}, which shows
      $\mathcal{N}_{\alpha}(\ell,t) \subseteq P(x,y,z)$ for some coprime integers $x \asymp y \asymp Y$ and $z \ll Z$. Clearly, 
      this shows $\#\mathcal{N}_{\alpha}(\ell,t) \ll Z^2$ and since $\mathcal{N}_{\alpha}(\ell,t) \subseteq [1,2^{\ell}]$, the pointwise estimate $\frac{N}{\varphi(N)} \ll \log \log N$ immediately shows the second case of \eqref{cases_bohr}.      
      It remains to prove the estimate in the case of $(\log Y)^{\xi} \leq Z$.
      Given an integer $N \in \mathcal{N}_{\alpha}(\ell,t)$, we see immediately that $N \leq Y^2$. We denote by
$N^*$ the $(\log Y)^{\xi/2}$-smooth part of $N$, 
    and observe that
    \[\frac{\frac{N}{\varphi(N)}}{\frac{N^{*}}{\varphi(N^*)}} 
    \ll \prod_{\substack{p \mid N\\ p \geq (\log Y)^{\xi/2}}}\left(1 + \frac{1}{p}\right)
    \leq \prod_{\substack{(\log Y)^{\xi/2} \leq p \leq 2 \log Y}}\left(1 + \frac{1}{p}\right)  \cdot \prod_{\substack{p \mid N\\ p \geq \log N}}\left(1 + \frac{1}{p}\right) \ll_{\xi} 1,
    \]
    by Mertens' estimate and the fact that $\#\{p \mid N: p \geq \log N\} \ll \frac{\log N}{\log \log N}$.
This allows us to replace $\frac{N}{\varphi(N)}$ by ${\frac{N^{*}}{\varphi(N^*)}}$ at a cost of only a constant factor.
We will use the standard identity
  \begin{equation}\label{mobius}
    \frac{n}{\varphi(n)}= \sum_{d \mid n}\frac{\mu^2(d)}{\varphi(d)},
  \end{equation}
  where $\mu$ denotes the M\"obius function.
Note that $\mu^2(d) = 1$ and $d$ $(\log Y)^{\xi/2}$-smooth implies for $Y$ sufficiently large
\[d \leq (\log Y)^{\xi/2}! \leq \exp((\log Y)^{\xi}).\]
Combining this with \eqref{mobius},
  we get
  \[
     \sum_{N \in \mathcal{N}_{\alpha}(\ell,t) }
     \frac{N}{\varphi(N)} \ll_{\xi} 
    \sum_{N \in P(x,y,z)}\frac{N^{*}}{\varphi(N^{*})} = \sum_{\substack{d \leq zy\\ p^{+}(d) \leq (\log Y)^{\xi}}}
    \frac{\mu^2(d)}{\varphi(d)}\sum_{\substack{N \in P(x,y,z)\\d \mid N}}1
    \leq \sum_{\substack{d \leq \exp((\log Y)^{\xi})}}\frac{\mu^2(d)}{\varphi(d)}\sum_{\substack{n \in P(x,y,z)\\d \mid n}}1.
  \]
  Following \cite[Proof of Lemma 7.4]{HK} verbatim, we obtain for squarefree $d$ and using the coprimality of $x$ and $y$ that
\[
    \sum_{\substack{n \in P(x,y,z)\\d \mid n}}1 \ll \frac{z^2}{d} + O(z).
  \]
  Summing over $d$ gives

  \[ \sum_{N \in \mathcal{N}_{\alpha}(\ell,t) }
     \frac{N}{\varphi(N)} \ll_{\xi} Z^2\sum_{d \leq \exp((\log Y)^{\xi})}\frac{\mu^2(d)}{d\varphi(d)} +  
     Z \sum_{d \leq \exp((\log Y)^{\xi})}\frac{\mu^2(d)}{\varphi(d)} \ll Z^2 + Z (\log Y)^{\xi}.
     \]
     Since by assumption, $(\log Y)^{\xi} \leq Z$, the result follows.
  \end{proof}

\subsection{Step 3: Variance estimate}\label{sec_var}
We have now gathered all estimates to prove \eqref{qia}, i.e., if 
\begin{equation}\label{meas_acc}\sum_{X \leq k \leq Y}\mu_k\psi_k \in [1,2],\end{equation}
then
    \begin{equation}\label{qia_3}
        \sum_{X \leq k,\ell \leq Y}\sum_{i \in D_k}\sum_{j \in D_{\ell}}\lambda(A_i \cap A_j) \leq C.
    \end{equation}
To that end, let $\varepsilon,\delta,\rho$ be as in Theorem \ref{thm_prime_like}. We now fix a pair $(k,\ell) \in [X,Y]^2$ with $k \leq \ell$ and provide a case distinction depending on the relative sizes of $k,\ell$:

  \begin{itemize}
      \item 
 Case 1: $\ell \geq \frac{3}{\varepsilon}\cdot k$:
   Here we will use the discrepancy bound (III). For this, we define for fixed $i \in D_k$ the slightly enlarged interval $A_i^+ := [i\alpha - \psi_k - \psi_{\ell},i\alpha + \psi_k + \psi_{\ell}] \supseteq A_i$, and observe that for $j \in D_{\ell}$, we have
   \[A_j \cap A_i \neq \emptyset \implies \{j\alpha\} \in A_i^+,\]
as well as $\lambda(A_i \cap A_{j}) \leq \psi_{\ell}.$
Thus for fixed $i \in D_k$, we get   
\[\sum_{j \in D_{\ell}}\lambda(A_i \cap A_{j}) \leq \psi_{\ell}\cdot \#\{j \in D_{\ell}: \{j\alpha\} \in A_i^+\}.\]

Using (III), we have
\[ \#\{j \in D_{\ell}: \{j\alpha\} \in A_i^+\}
= \lambda(A_i^+)\mu_{\ell} + O(\mu_{\ell}^{1-\varepsilon}).
\]
By (II), we have $\mu_{\ell} \gg \frac{2^{\ell}}{\ell}$
and since by \eqref{wlog_i}, we have $\lambda(A_i^+) = \psi_k \geq \frac{1}{2^{2k}} \geq \frac{1}{2^{(2\varepsilon/3)\ell}} \gg \mu_{\ell}^{-\varepsilon}$,
we get
\[\#\{j \in D_{\ell}: \{j\alpha\} \in A_i^+\}\ll \lambda(A_i^+)\mu_{\ell}.\]

Summing over $i \in D_{k}$ shows
    \begin{equation}\label{wish_var} \sum_{i \in D_k}\sum_{j \in D_{\ell}}\lambda(A_i \cap A_{j}) \ll \psi_k\psi_{\ell}\mu_k\mu_{\ell}.\end{equation}

    \item Case 2: $\ell < \frac{3}{\varepsilon}\cdot k$. %Note that using again \eqref{remove_restr} and
    Since the intersection of two intervals is either disjoint or its measure is bounded by the size of the smaller one, we have
    \[\lambda(A_i \cap A_{j}) \leq \psi_{\ell}\mathds{1}_{[\lVert (i - j)\alpha\rVert \leq \psi_k]},\quad 
    i \in D_k, j \in D_{\ell}.
    \]
    
    Thus we get by using Proposition \ref{sieve_prop}
    \[\begin{split}\sum_{i \in D_k}\sum_{j \in D_{\ell}}\lambda(A_i \cap A_{j})
    &\ll \psi_{\ell}\cdot \#\left\{i \in D_k, j \in D_{\ell}: |j-i| \in \mathcal{N}_{\alpha}({\ell},\psi_k)\right\}
    \\& \leq \psi_{\ell}\sum_{h \in \mathcal{N}_{\alpha}({\ell},\psi_k)}\#\{n \in D_k: n + h \in \A\}
    \\& \ll_{\rho} \psi_{\ell}\frac{2^{k}}{d_k^2}\sum_{h \in \mathcal{N}_{\alpha}({\ell},\psi_k)}\frac{h}{\varphi(h)}
    \\&\ll_{\varepsilon}\psi_{\ell}\frac{2^{k}}{d_kd_{\ell}}\sum_{h \in \mathcal{N}_{\alpha}(2^{\ell},\psi_k)}\frac{h}{\varphi(h)},
    \end{split}
    \]
    where we used in the last line that $d_{\ell}/d_k \ll_{\varepsilon} 1$ since $\ell \leq \frac{3}{\varepsilon}\cdot k$.\\
    
    Let $Z_{\ell,k} := \sqrt{2^{\ell}\psi_k}$ and $Y_{\ell,k} := \sqrt{\frac{2^{\ell}}{{\psi_k}}}$.
    We will now apply Lemma \ref{bohr_average} in its different cases depending on the relative sizes of $Z_{\ell,k}$ and $Y_{\ell,k}$.
    \begin{itemize}
    \item Case 2a: $Z_{\ell,k} < \sqrt{c(\alpha)}$: In this case, Lemma \ref{bohr_average} shows that
    $\mathcal{N}_{\alpha}(2^{\ell},\psi_k) = \emptyset$, so \eqref{wish_var} holds trivially for such $k,\ell$.
    \item Case 2b: $\sqrt{c(\alpha)} \leq Z_{\ell,k}, (\log Y_{\ell,k})^{\delta/3} \leq Z_{\ell,k} $. In this case, Lemma \ref{bohr_average} shows

\[\sum_{i \in D_k}\sum_{j \in D_{\ell}}\lambda(A_i \cap A_{j}) \ll_{\delta} \psi_{\ell}\frac{2^{k}}{d_kd_{\ell}}Z_{\ell,k}^2 \ll \mu_k\psi_k \mu_{\ell}\psi_{\ell}.
\]
    \item Case 2c: $\sqrt{c(\alpha)} < Z_{\ell,k} < (\log Y_{\ell,k})^{\delta/3}$.
    We write $\ell^{-} = \ell^{-}(k),\ell^{+} = \ell^{+}(k)$ for the biggest respectively smallest value of $\ell$ that satisfies $\sqrt{c(\alpha)} < Z_{\ell,k} < \log(Y_{\ell,k})^{\delta/3}, \ell \leq \frac{3}{\varepsilon}\cdot k$. 
    
Since $\ell \leq \frac{3}{\varepsilon}\cdot k$, we have $\log \ell \ll_{\varepsilon} \log k$.
    Using $\psi_{\ell} \leq \psi_{k}$ since $k \leq \ell$, we thus get from Lemma \ref{bohr_average} that
\[\begin{split}
\sum_{i \in D_k}\sum_{\ell^{-}\leq \ell \leq \ell^{+}}\sum_{j \in D_{\ell}}\lambda(A_i \cap A_{j})
&\ll \log k \frac{2^{k}}{d_k^2}\sum_{\ell^{-}\leq \ell \leq \ell^{+}} \psi_{\ell}Z_{\ell,k}^2
\\&\ll\log k \frac{2^{k}}{d_k^2}\psi_{k}^ 2\sum_{\ell^{-}\leq\ell \leq \ell^{+}}  2^{\ell}
\\&\ll\log k \frac{2^{k}}{d_k^2}\psi_k Z_{\ell^{+},k}^2 \\&\ll \log k \frac{2^{k}}{d_k^2}\psi_k(\log Y_{\ell^{+},k})^{2\delta/3}.
    \end{split}\]
    Note that $Y_{\ell,k} \leq 2^{\ell}$ and hence $\log Y_{\ell^+,k} \ll_{\varepsilon} k$. Further, by (II), we have
    $d_k \gg k^{\delta}$. Consequently, $d_k \gg (\log Y_{\ell^{+},k})^{2\delta/3} \log k$, and therefore,

\[\sum_{i \in D_k}\sum_{\ell^{-}\leq \ell \leq \ell^{+}}\sum_{j \in D_{\ell}}\lambda(A_i \cap A_{j})
\ll \frac{2^k}{d_k}\psi_k = \mu_k\psi_k.
\]
      \end{itemize}
\end{itemize}
    Combining all the estimates above, we obtain for any fixed $X \leq k \leq Y$,

    \[\sum_{i \in D_k}\sum_{k \leq \ell < Y}\sum_{j \in D_{\ell}}\lambda(A_i \cap A_{j})
    \ll \mu_k\psi_k \left(1 + \sum_{k \leq \ell < Y}\mu_{\ell}\psi_{\ell}\right)
    \ll \mu_k\psi_k \sum_{X \leq \ell < Y}\mu_{\ell}\psi_{\ell},
    \]
    where in the last line we used assumption \eqref{meas_acc}.
    Summing over $k$, this finally shows \eqref{qia_3}.
    We are now in the position to apply Lemma \ref{bdv_applied}. This shows that

\[ BAD \subseteq \bigcap_{\psi \in \mathcal{M}_{\A}}K_{\A}(\psi).\]

\subsection{Step 4: Counterexamples for well-approximable numbers}\label{counterex_sec}

In this section, we prove the remaining set inclusion of Theorem \ref{thm_prime_like}, that is, 
\begin{equation}\label{counter_1}\bigcap_{\psi \in \mathcal{M}_{\A}}K_{\A}(\psi) \subseteq BAD_{\p}.\end{equation}
Before proving this, we explain how this implies the ``in particular'' statement of Theorem \ref{thm_prime_like}.
Note that \[BAD_{\p} \subseteq BAD_{\mathbb{P}} \subseteq \left\{\alpha \in \mathbb{R}: \lVert n\alpha \rVert \leq \frac{1}{n \log n \log \log n} \text{ only finitely often}\right\}.\]
However, by Khintchine's Theorem applied with $\psi(n) = \frac{1}{n \log n \log \log n}$, we have (see \eqref{khin_condition}), 
$\lambda(W(\psi)) = 1$. Consequently, since $\bigcap_{\psi \in \mathcal{M}_{\A}}K_{\A}(\psi) \subseteq BAD_{\p} \subseteq [0,1]\setminus W(\psi)$, it follows that 
$\lambda(\bigcap_{\psi \in \mathcal{M}_{\A}}K_{\A}(\psi)) = 0$. This will show the ``in particular'' statement of Theorem \ref{thm_prime_like}, since by $BAD \subseteq \bigcap_{\psi \in \mathcal{M}_{\A}}K_{\A}(\psi)$ and $\dim_H(BAD) = 1$, we have that $\bigcap_{\psi \in \mathcal{M}_{\A}}K_{\A}(\psi)$ has full Hausdorff dimension.\\

The statement provided below will not only show \eqref{counter_1}, but will be general enough to also deduce 
$\bigcap_{\psi \in \cM_{\A}}\mathcal{K}_{\A}(\psi) \subseteq BAD$ in the setup of Theorem \ref{thm_pos_dens}.
Before we prove this, let us define a family of functions that will serve the purpose of (approximate) density functions.
Let $f : \mathbb{N} \to [0,\infty)$ be a monotonically increasing function satisfying
\begin{equation}\label{pseudo_log}
f(x^n) \ll_n f(x),\quad x \geq 1.
\end{equation}
We remark that \eqref{pseudo_log} is in particular satisfied for 
$f(x) = (\log x)^C$ for all $C \geq 0$.
With this defined, let us state the main result of this section.

\begin{prop}\label{counterex_non_bad}
    Let $\A \subseteq \mathbb{N}$ be a set that satisfies
    \begin{equation}\label{dens_f}
        \liminf_{N \to \infty} \frac{\#\{n \leq N: n \in \A\}f(N)}{N} > 0,
    \end{equation}
    where $f$ is a function satisfying \eqref{pseudo_log}. Then for every real $\alpha$ where 
    \begin{equation}\label{well_approx_log}\liminf_{q \to \infty} qf(q)\lVert q\alpha \rVert = 0,\end{equation}
    there exists $\psi \in \mathcal{M}_A$ such that 
    $\alpha \notin \mathcal{K}_{\A}(\psi)$.
\end{prop}

Before we prove the statement, let us quickly show that this implies \eqref{counter_1}:
Indeed, this follows by setting $f(x) := f_{\mathcal{P}}(x)$ and observing that

\[\frac{f_{\mathcal{P}}(x^n)}{f_{\mathcal{P}}(x)}
= \prod_{\substack{x < \leq p \leq x^n\\ p \in \mathcal{P}}} \left(1 + \frac{1}{p}\right) \leq \prod_{\substack{x < \leq p \leq x^n\\ p \in \mathbb{P}}} \left(1 + \frac{1}{p}\right) \ll_n 1,
\]
by Mertens' estimate, which shows \eqref{pseudo_log}. Hence, we may apply Proposition \ref{counterex_non_bad}, and \eqref{counter_1} follows. Thus we are left to show Proposition \ref{counterex_non_bad}.

\begin{proof}
    This builds upon the strategy of Kurzweil's original article \cite{kurzweil}. 
   By dividing $f$ by a suitable constant, we may assume without loss of generality that $\liminf_{N \to \infty} \frac{\#\{n \leq N: n \in \A\}f(N)}{N} \geq 1$ -- note that this will not change the assumption of \eqref{well_approx_log}.
    Using \eqref{well_approx_log}, we find a sequence of convergent denominators $q_{n_k} = q_{n_k}(\alpha)$ such that $\lVert q_{n_k}\alpha\rVert \leq \frac{4^{-k}}{q_{n_k} f(q_{n_k})}$. We may choose the sequence $(n_k)_k$ arbitrarily sparse, which will be specified later.
    Next, we define sets $(S_k)_{k \in \N}$ with 
    \[S_0 = \emptyset, \quad S_k = [1,2^{k}q_{n_k} f(q_{n_k})] \setminus \bigcup_{\ell < k}S_{\ell}, \quad k \geq 1,\] and assume $q_{n_k}$ so sparse that
    $\# (S_k \cap \mathcal{A}) > \frac{\max S_k}{2 f(\max S_k)}$.  We will now define $\psi_{|_{S_k}} \equiv  \psi_k :=\frac{2^{-k}}{q_{n_k}}$ and claim that $\psi \in \mathcal{M}_{\A}$: Clearly, $\psi$ is monotonically decreasing, so it remains to prove $\sum_{n \in \A} \psi(n) = \infty.$ We observe that by \eqref{dens_f},
\[\sum_{n \in \A} \psi(n) = \sum_{k \geq 1} \psi_k \# (S_k \cap \mathcal{A})
\geq \frac{1}{2} \sum_{k \geq 1} \psi_k \frac{\#S_k}{f(\max S_k)} 
\gg \sum_{k \geq 1} \frac{2^{-k}}{q_{n_k}} \frac{2^k q_{n_k} f(q_{n_k})}{f(2^kq_{n_k}f(q_{n_k}))}.
\]
Since we may choose $q_{n_k} > 2^k$ by picking $(n_k)_k$ sparse enough, and $f(x) \leq x$ for $x$ sufficiently large, we have by \eqref{pseudo_log} that for $k$ sufficiently large,
\[f(2^kq_{n_k}f(q_{n_k})) \leq f(q_{n_k}^3) \ll f(q_{n_k}).\] Consequently, 
we obtain $\sum_{n \in \A} \psi(n) \gg \sum_{k \geq 1} 1 = \infty,$ proving $\psi \in \mathcal{M}_{\A}$.\\

On the other hand, we will show that we have
\begin{equation}
    \label{conv_BC}
    \lambda\left(\bigcup_{n \in S_k \cap \A} A_n\right) \ll 2^{-k}, \quad k \geq 1,
\end{equation}
which by the convergence Borel--Cantelli Lemma will immediately prove $\alpha \notin \mathcal{K}_{\A}(\psi)$.\\

In order to show \eqref{conv_BC}, we observe that since $\lVert q_{n_k}\alpha\rVert \leq \frac{4^{-k}}{q_{n_k} f(q_{n_k})}$, for all $j \leq 2^kf(q_{n_k})$ and all $0 \leq m \leq q_{n_k}-1$ that
\begin{equation}\label{error-pkqk}(m + jq_{n_k})\alpha \equiv m \frac{p_{n_k}}{q_{n_k}} + \left(m + jq_{n_k}\right)\left(\alpha - \frac{p_{n_k}}{q_{n_k}}\right)
\equiv \frac{a}{q_{n_k}} + O\left(\frac{2^{-k}}{q_{n_k}} \right) \pmod 1,
\end{equation}
where $a = a(m) := p_{n_k}m \pmod {q_{n_k}}$. This implies that for any $0 \leq m \leq q_{n_k}-1$,
\[
\bigcup_{\substack{n \in \A \cap S_k\\ n \equiv m \pmod {q_{n_k}}}} A_{n} \subseteq \left[\frac{a}{q_{n_k}} - \psi_k -  O\left(\frac{2^{-k}}{q_{n_k}} \right), \frac{a}{q_{n_k}} + \psi_k + O\left(\frac{2^{-k}}{q_{n_k}} \right)\right] \pmod 1,
\]
thus (recall $\psi_k = \frac{2^{-k}}{q_{n_k}}$),
\[\lambda\Big(\bigcup_{\substack{n \in \A \cap S_k\\ n \equiv m \pmod {q_{n_k}}}} A_{n}\Big)
\ll \frac{2^{-k}}{q_{n_k}}.
\]
Summing over $m = 0,\ldots,q_{n_{k}}-1$, and applying the union bound, proves 
\eqref{conv_BC}, finishing the statement.
\end{proof}

As mentioned in the Introduction and can be seen in the proof of Proposition \ref{counterex_non_bad}, we do not make use of any particular structure of $\A$ to deduce \eqref{counter_1}, besides an estimate on the density. This might be optimal for some sequences $\A$; however, if we use a multiplicative structure as in Theorem \ref{thm_prime_like}, we can actually find counterexamples that are slightly worse approximable than the rate provided by the rare density argument used in Proposition \ref{counterex_non_bad}. As an illustration, we show this below for the case of $\A = \mathbb{P}$.
Recall that for a function $f$, we defined $BAD_f := \{\lim_{n \to \infty} n f(n)\lVert n\alpha\rVert > 0\}$.

\begin{prop}\label{exploit_prime_counterex}
    For any $\varepsilon > 0$, there exists $\alpha \in BAD_{\log/(\log \log \log)^{1-\varepsilon}}$ such that 
    $\alpha \notin \bigcap_{\psi \in \mathcal{M}_{\mathbb{P}}}{\mathcal{K}}_{\mathbb{P}}(\psi)$.
\end{prop}

\begin{proof}
    We first construct $\alpha \in BAD_{\log/(\log \log \log)^{1-\varepsilon}}$ such that for infinitely many
    $q_k$ that lead to good approximations, we have 
    $\frac{\varphi(q_k)}{q_k} \ll (\log \log \log q_k)^{-1}$. We do this inductively for $k \geq 0$ along a subsequence $(j_k)_k$ by constructing $3$ suitable consecutive partial quotients, that will provide for $n_k = j_{k}+2$ the following estimates:
    
    \begin{align}
    \label{good_approx}\lVert q_{n_k}\alpha\rVert &\leq \frac{4^{-k}}{q_{n_k}(\log \log \log q_{n_k})^{1-\varepsilon/2}},\\
    \label{small_phi} \frac{\varphi(q_{n_k})}{q_{n_k}} &\ll (\log \log \log q_{n_k})^{-1}.
    \end{align}  
    Having an initial segment $(q_n)_{n \leq k}$ constructed, we have different candidates for $a_{k+1}$, and we will choose a suitable one: Let 
    $q_{k+1}(a) := aq_k + q_{k-1}$.
    We now choose $a$ such that for all small primes $p$, we get
    $p \nmid q_{k+1}(a)$. More precisely, let $\p(q_k) := \{p \in \mathbb{P}: p \leq (\log \log q_k)/2\}$.
    We claim that the congruence system
    \begin{align}\label{cong_sys1}
        q_{k+1}(a) = aq_k + q_{k-1} \not \equiv 0 \pmod p, \quad \forall p \in \p(q_k)
    \end{align}
    has a solution $a \asymp \log q_k/(\log \log \log q_k)^{1-\varepsilon}$. Indeed, observe that for every $p \in \p(q_k)$ where $p \mid q_k$, we have $q_{k+1}(a) \equiv q_{k-1} \not \equiv 0 \pmod p$ since $q_{k-1},q_k$ are coprime.
    If $p \nmid q_k$, then $(a q_k)_{a = 0}^{p-1}$ runs through all residued classes mod $p$, thus in particular there exists one $a \in \mathbb{Z}_p$ such that $ aq_k + q_{k-1}  \equiv 1 \pmod p$.
    Thus by the Chinese Remainder Theorem, there exists at least one solution for \eqref{cong_sys1} mod $\prod_{p \in \p(q_k)}p$. Note that $\prod_{p \in \p(q_k)}p
    \leq 4^{\log \log q_k/2} \ll \log q_k /\log \log q_k$ by a standard estimate on the primorial. This allows us to pick 
    $a_0 = a_0(k) \asymp \log q_k/(\log \log \log q_k)^{1-\varepsilon}$ such that \eqref{cong_sys1} holds.
    \\

    Having $q_{k+1} = q_{k+1}(a_0)$ fixed, we now provide a similar procedure for $q_{k+2}$ in place of $q_{k+1}$. 
    However, in this case, we aim for
       \begin{align}\label{cong_sys2}
        q_{k+2}(b) = bq_{k+1} + q_{k} \equiv 0 \pmod p, \quad \forall p \in \p(q_k).
    \end{align}
By construction, $q_{k+1}$ is coprime to any $p \in \p(q_k)$ (this was in fact the whole point of constructing $a_0$). Thus another application of the Chinese Remainder Theorem
allows us to find $b_0 = b_0(k) \asymp \log q_k/(\log \log \log q_k)^{1-\varepsilon}$ such that \eqref{cong_sys2} has a solution.\\

Finally, we let $c_0 = c_0(k) \asymp \log q_k/(\log \log \log q_k)^{1-\varepsilon}$ be an arbitrary integer
and set $q_{k+3} = c_0q_{k+2} + q_{k+1}$.

We now apply this procedure iteratively along an arbitrarily sparse subsequence $(j_k)_k$, filling the other partial quotients with $1$'s. This means that we construct \[\alpha = [0;1,1,\ldots,1,a_0(j_1),b_0(j_1),c_0(j_1),1,\ldots,1,a_0(j_2),b_0(j_2),c_0(j_2)1,\ldots].\]

We claim that $\alpha \in BAD_{\log/(\log \log \log)^{1-\varepsilon}}$: Indeed, with $\alpha = [0;a_1,a_2,\ldots]$, we have
$q_n \lVert q_n \alpha \rVert \asymp \frac{1}{a_{n+1}} \gg \frac{(\log \log \log q_n)^{1-\varepsilon}}{\log q_n}$,
which is sharp when $n = j_k+2$, proving \eqref{good_approx} by assuming $j_k$ to be sufficiently large.
On the other hand, we have
\[\varphi(q_{j_k+2})/q_{j_k+2} \asymp \prod_{p \leq \log \log q_{j_k}/2} \left(1 - \frac{1}{p}\right)
\ll \frac{1}{\log \log \log q_{j_k+2}},
\]
by Mertens' estimate, proving \eqref{small_phi}.\\

Following the proof of Proposition \ref{counterex_non_bad}, we get by \eqref{good_approx} and, if necessary, choosing the sequence $(n_k)_k$ sparser,
that $\lVert q_{n_k}\alpha\rVert \leq \frac{4^{-k}}{q_{n_k}f(q_{n_k})}$ 
with $f(q_{n_k}) = \frac{\log q_n}{(\log \log \log q_n)^{1-\varepsilon/2}}$. We define $S_k$ and $\psi$ as in the proof of Proposition \ref{counterex_non_bad}, which proves immediately that
$\sum_{p \in \mathbb{P}}\psi(p) = \infty.$

The crucial point is to prove \eqref{conv_BC}, which would fail if we dropped the condition on $\A$ completely, since
\eqref{error-pkqk} gets replaced with the weaker error term $O\left(\frac{2^{-k} (\log \log \log q_{n_k})^{1-\varepsilon}}{q_{n_k}} \right)$, and the aim is to win back the factor $(\log \log \log q_{n_k})^{1-\varepsilon}$.
For this, note that if $m + jq_{n_k} \in \mathbb{P}$ with $j \geq 2$, then $a(m) := p_{n_k}m \pmod {q_{n_k}}$ is coprime to $q_{n_k}$:
Indeed, assume for a prime $p_0$ that 
\[p_0 \mid \gcd(a(m),q_{n_k}) = \gcd(p_{n_k}(m+jq_{n_k}),q_{n_k}) = \gcd((m+jq_{n_k}),q_{n_k}),\] where we used that $q_{n_k},p_{n_k}$ are coprime. Since $m+jq_{n_k} \in \mathbb{P}$, this implies $p_0 = m+jq_{n_k}$ and $p_0 \mid q_{n_k}$. In particular, $m+jq_{n_k} \leq q_{n_k}$, a contradiction to $j \geq 2$.
Thus we get
\[\begin{split}&\lambda\left(\bigcup_{n \in S_k \cap \mathbb{P}} A_p\right) \ll 
\lambda\left(\bigcup_{\substack{p \in S_k \cap \mathbb{P}\\p \leq 2q_{n_k}}} A_p\right)
+ \\&\lambda\left( \bigcup_{\substack{0 \leq j \leq 2^k\log q_{n_k}+1\\(a,q_{n_k}) = 1}}\left[\frac{a}{q_{n_k}} - \frac{2^{-k}}{q_{n_k}}- O\left(\frac{4^{-k} (j\log \log \log q_{n_k})^{1-\varepsilon/2}}{q_k \log q_{n_k}}\right),\frac{a}{q_{n_k}} + \frac{2^{-k}}{q_{n_k}}+ O\left(\frac{4^{-k} (j\log \log \log q_{n_k})^{1-\varepsilon/2}}{q_k \log q_{n_k}}\right) \right]\right)
\\&\ll \frac{2^{-k}}{\log q_{n_k}} + \frac{\varphi(q_{n_k})}{q_{n_k}}(\log \log \log q_{n_k})^{1-\varepsilon/2}
\ll 2^{-k},
\end{split}
\]
where we used \eqref{small_phi} in the last line. This shows \eqref{conv_BC}, and the result follows.
\end{proof}

\section{Proof of Corollaries \ref{prime_cor} - \ref{intersec_cor}}\label{subsec_prime_s2}

We will prove Corollaries \ref{prime_cor} - \ref{intersec_cor} for certain subsets of the original sets $\bS_2, \mathbb{L}, \bS_2 \cap \mathbb{L}$ that will be of positive relative density, but with a cleaner sieve setup (see \eqref{sieve_s2} - \eqref{sieve_i}) below) that allows us to deduce property (III).
By the following result, this suffices to prove the actual corollaries:

\begin{prop}\label{pos_rel_dens_suff}
    Let $\A' \subseteq \A$ with $\liminf_{N \to \infty} \frac{\#\{n \leq N: n \in \A'\}}{\#\{n \leq N: n \in \A\}} > 0$. 
    Then 
    \[\bigcap_{\psi \in \mathcal{M}_{\A'}}K_{\A'}(\psi) \subseteq \bigcap_{\psi \in \mathcal{M}_{\A}}K_{\A}(\psi).\]
\end{prop}

\begin{proof}
    We first show $\mathcal{M}_{\A} = \mathcal{M}_{\A'}$. By enumerating $\A = \{a_1 < a_2 < \ldots\}$ and setting $\theta(n) := \psi(a_n)$, we see that it suffices to show this for the case where $\A = \N$, which implies that $\A'$ has positive lower density (in $\mathbb{N}$).
    Clearly, since $\A' \subseteq \N$, we have $\mathcal{M}_{\A'} \subseteq \mathcal{M}$.
    On the other hand, a straightforward computation using summation by parts and the positive lower density shows that for $\psi \in \mathcal{M}$, we have $\sum_{n \in \A'} \psi(n) = \infty$, proving $\mathcal{M}_{\A} = \mathcal{M}_{\A'}$.
    Now fix $\psi \in \mathcal{M}_{\A} = \mathcal{M}_{\A'}$. By definition, using only $\A' \subseteq \A$, we get
    \[\begin{split}T_{\A'}(\psi,\alpha) &= \{\gamma \in [0,1): \lVert n\alpha + \gamma \rVert \leq \psi(n) \text{ for infinitely many } n \in \A'\} \\&\subseteq \{\gamma \in [0,1): \lVert n\alpha + \gamma \rVert \leq \psi(n) \text{ for infinitely many } n \in \A\} = T_{\A}(\psi,\alpha).\end{split}\]
    Since by definition, $K_{\A}(\psi) := \{\alpha \in [0,1): \lambda(T_{\A}(\psi,\alpha)) = 1\}$, this immediately proves
    $K_{\A'}(\psi) \subseteq K_{\A}(\psi)$.
\end{proof}

Having Proposition \ref{pos_rel_dens_suff} established, we now recall some classical properties of the sets $\bS_2, \mathbb{L}, \bS_2 \cap \mathbb{L}$.
Defining $N_{a,b} := \{N \in \mathbb{N}: p \mid N \implies p \equiv a \pmod b\}$, we have (see e.g. \cite[Proposition 6.2]{waldschmidt}):

\begin{itemize}
    \item $\bS_2 = \{n \in \mathbb{N}: n = 2^k N_{3,4}^2 N_{1,4}\}$.
    \item $\mathbb{L} =  \{n \in \mathbb{N}: n = 3^k N_{2,3}^2 N_{1,3}\}$.
    \item $\mathbb{I} := \mathbb{L} \cap \bS_2 = \{n \in \mathbb{N}: n = (2^m3^k N_{5,12}N_{7,12}N_{11,12})^2 N_{1,12}\}$.
\end{itemize}
In turn, we define
    \[\bS_2' := N_{1,4} \subset \bS_2, \quad \mathbb{L}' :=  N_{1,3} \subset \mathbb{L}, \quad \mathbb{I}' = N_{1,12} \subset \mathbb{I}.\]
We remark that these sets have themselves a meaning as primitive representations, e.g. $\bS_2'$ being the set of all integers that can be represented as the sum of two coprime odd integers, and similar characterizations hold for $\mathbb{L}', \mathbb{I}'$. We do not pursue this further, but recall the following properties that follow standard considerations in analytic number theory.

\begin{lem}\label{lem_densities}
    Let all quantities be defined as above. Then we have the following:
\begin{itemize}[noitemsep, topsep=0pt, parsep=0pt, partopsep=0pt]
    \item[(i)] For any $N \in \mathbb{N}$, we have 
    \begin{align}
       \label{sieve_s2}\bS_2' \cap [0,N] &= \left\{n \leq N: n \equiv 1 \pmod 4, p \nmid n \quad \forall p \equiv 3 \pmod 4 \text{ where } p \leq \sqrt{N}\right\}\\
       \label{sieve_l}\mathbb{L}' \cap [0,N] &= \left\{n \leq N: n \equiv 2 \pmod 3, p \nmid n \quad \forall p \equiv 2 \pmod 3 \text{ where } p \leq \sqrt{N}\right\}\\
    \label{sieve_i}\mathbb{I}' \cap [0,N] &= \left\{n \leq N: n \equiv 1 \pmod {12}, p \nmid n \quad \forall p \equiv 5,7,12 \pmod {12} \text{ where } p \leq \sqrt{N}\right\}.
    \end{align}
    
    \item[(ii)] There exist constants $c_1,c_2,c_3,c_1',c_2',c_3' > 0$ such that 
    \begin{align}\label{asymp_s2}\bS_2 \cap [0,N] &\sim c_1 \frac{N}{\sqrt{\log N}},\quad \mathbb{L} \cap [0,N] \sim c_2 \frac{N}{\sqrt{(\log N)}}, \quad
\mathbb{I} \cap [0,N] \sim c_3 \frac{N}{(\log N)^{3/4}},\\
\bS_2' \cap [0,N] &\sim c_1' \frac{N}{\sqrt{\log N}},\quad \mathbb{L}' \cap [0,N] \sim c_2' \frac{N}{\sqrt{(\log N)}}, \quad
\mathbb{I}' \cap [0,N] \sim c_3' \frac{N}{(\log N)^{3/4}}.\label{asymp_s2'}
\end{align}
\end{itemize}
\end{lem}
\begin{proof}
\item
\begin{itemize}
    \item[(i)]
    This is explained in \cite{ramare}, and builds on the parity condition (see \cite[Section 14.4]{opera}), but we provide the proof for $\bS_2'$ here for completeness; the proofs for $\mathbb{L}'$ and $\mathbb{I}'$ work analogously.
    By definition, we have that the left-hand side of \eqref{sieve_s2} is contained in the right-hand side. Thus it suffices to show that if $n \leq N$ and $n \equiv 1 \pmod 4, p \nmid n \forall p \equiv 3 \pmod 4$ where $p \leq \sqrt{N}$, then $n \in \bS_2'$, that is, we also have $p \nmid n \forall p \equiv 3 \pmod 4$ where $\sqrt{N} < p \leq N$.
    Assuming by contradiction that there is such a $\sqrt{N} < p \leq N$ with $p \mid n$, then since $p \equiv 3 \pmod 4$ and $n \equiv 1 \pmod 4$, there must be $p' \equiv 3 \pmod 4$ such that $pp' \mid n$. Since $p > \sqrt{N}$, we have $p' \leq \sqrt{N}$ with $p'\mid n$ and $p' \equiv 3 \pmod 4$, a contradiction.
    \item[(ii)]
    This can be proven by the Lagrange--Selberg--Delange method (see \cite[Proposition 6.1]{waldschmidt} for how to obtain \eqref{asymp_s2}). The same proof could be applied to obtain \eqref{asymp_s2'}; an even more direct proof with \eqref{asymp_s2} established is to observe that for the Dirichlet series $F_{\A}(s) := \sum_{n \in \A} \frac{1}{n^s}$, we have that for $Re(s) > 1/2$ that

    \[\frac{F_{\bS_2}(s)}{F_{\bS_2'}(s)} = \left(\frac{1}{1 - 2^{-s}}\right)\prod_{p \equiv 3 \pmod 4}\left(\frac{1}{1 - p^{-2s}}\right).\]
    Thus the limit $s \to 1^+$ is completely unproblematic, showing
    \[c_1/c_1' = \left(\frac{1}{1 - 2^{-1}}\right)\prod_{p \equiv 3 \pmod 4}\left(\frac{1}{1 - p^{-2}}\right).\]
    The same argument can be adapted straightforwardly to $\mathbb{L}',\mathbb{I}'$.
    \end{itemize}
\end{proof}
Note that $\bS_2',\mathbb{L}'\mathbb{I}'$ are already in the form that they satisfy (I) and (II) of Theorem \ref{thm_prime_like}, and the same holds true for $\mathbb{P}$ for trivial reasons. In order to establish (III), we make use of exponential sum estimates due to Vaughan \cite{vaughan_primes} as well as the recent work of Ramar\'e and Viswanadham \cite{ramare}.

\begin{lemma}\label{sum_expsum}
     Let $R,N > 0$ and $\xi > 0$ be arbitrary. Let
     $a,q \in \mathbb{N}$ such that $\gcd(a,q) = 1, |\alpha - a/q| \leq \tfrac{1}{q^2}, q \leq RN$. Let $\mathcal{P}$ be a set of primes and let
    $\cS_N := \{n \leq N: p \nmid n \forall p \in \mathcal{P}\cap [1,\sqrt{N}]\}$.
   Then \begin{equation}\label{ramare_exp_gen}\sum_{r \leq R}\left\lvert \sum_{\substack{n \in \cS_N}}e(rn\alpha)\right\rvert
\ll_{\xi} R^ 2 + RN\left(\frac{1}{\sqrt{q}} + \frac{\sqrt{q}}{\sqrt{RN}} + \frac{1}{R^{1/3}N^{1/6}} + \frac{R^{1/3}}{N^{1/3}}\right)N^{\xi}.
\end{equation}
\end{lemma}

\begin{proof}
The proof of Lemma \ref{sum_expsum} is a straightforward generalization of \cite[Theorem 1.8]{ramare}, where \eqref{ramare_exp_gen} was shown in the case $\mathcal{P} = \{p \in \mathbb{P}: p \equiv 3 \pmod 4\}$. For completeness, we provide a proof for the general case, which follows the argument of \cite{ramare} verbatim.
Let us define the notation of $n \approx N$ meaning $n \in (N/2,N]$.
    By a sieve identity \cite[Theorem 1.4]{ramare}, we obtain for $n \approx N, r \approx R$, and any values $2 \leq z \leq \sqrt{N}$ and
    $M \geq 2$ optimized later that
    
    \[\begin{split}\sum_{r \approx R}\left\lvert \sum_{\substack{n \in \cS_N\\n \approx N}}e(rn\alpha)\right\rvert
    &\leq \sum_{r \approx R}\left\lvert\sum_{\substack{d \leq M\\d \mid \p(z)}}\mu(d)\sum_{\substack{n \equiv 0 \pmod d\\n \approx N}} e(rn\alpha)\right\rvert
    + \sum_{r \approx R}\left\lvert\sum_{\substack{mp \approx N\\z \leq p \leq \sqrt{N}\\p \in \p}}\rho(m)e(mpr\alpha) \right\rvert\\&+
    \sum_{r \approx R}\int_{0}^1 \left\lvert \sum_{\substack{k\ell \approx N\\M \leq \ell \leq Mz}}\alpha_{\ell}(t)\beta_k(t)e(k\ell r\alpha)\right\rvert \,\mathrm{d}t
     + O(RN/z),
     \end{split}
    \]
    where $\rho(m) \in [0,1]$, $|\alpha_{\ell}(t)| \leq 1, |\beta_k(t)| \leq N^{o(1)}$.\\

    An application of \cite[Lemma 10.3]{ramare}, which is a variant of Harman's \cite[Lemma 4]{harm_primes}, shows
    (after localizing dyadically and summing over dyadic ranges) that for any $\xi > 0$,

    \[\sum_{r \approx R}\Big\lvert \sum_{\substack{mp \approx N\\z \leq p \leq \sqrt{N}\\p \in \p}}\rho(m)e(mpr\alpha) \Big\rvert
    \ll_{\xi} NR\left(\frac{1}{\sqrt{q}} + \frac{\sqrt{q}}{\sqrt{RN}} + \frac{1}{N^{1/4}} + \frac{1}{\sqrt{Rz}}\right)N^{\xi}.
    \]
    Employing the same routine again, we get

    \[\sum_{r \approx R}\int_{0}^1 \Big\lvert\sum_{\substack{k\ell \approx N\\M \leq \ell \leq Mz}}\alpha_{\ell}(t)\beta_k(t)e(k\ell r\alpha)\Big\rvert \,\mathrm{d}t
    \ll_{\xi} NR\left(\frac{1}{\sqrt{q}} + \frac{\sqrt{q}}{\sqrt{RN}} + \frac{\sqrt{Mz}}{\sqrt{N}} + \frac{1}{\sqrt{RM}}\right)N^{\xi}.
    \]
    Applying \cite[Lemma 3]{vaughan_primes} (see also \cite[Lemma 10.1]{ramare}), we get after combining $(r,d)$ into one variable $u = rd$ that
    \[\Big\vert\sum_{r \approx R}\sum_{\substack{d \leq M\\d \mid \p(z)}}\mu(d)\sum_{\substack{n \equiv 0 \pmod d\\n \approx N}} e(rn\alpha)\Big\vert \ll_{\xi} \left(\frac{NR}{q} + RM + q\right)N^{\xi}.
    \]
We combine the above estimates and choose $M = z = (N/R)^{1/3}$, which gives for $q \leq NR$ the simplified estimate

\[\sum_{r \approx R}\left\lvert \sum_{\substack{n \in \cS_N\\n \approx N}}e(rn\alpha)\right\rvert
\ll_{\xi} RN\left(\frac{1}{\sqrt{q}} + \frac{\sqrt{q}}{\sqrt{RN}} + \frac{1}{R^{1/3}N^{1/6}} + \frac{R^{1/3}}{N^{1/3}}\right)N^{\xi},
\]
provided $N > R$. If $R > N$, then a trivial estimate gives 
\[\sum_{r \leq R}\left\lvert \sum_{\substack{n \in \cS_N}}e(rn\alpha)\right\rvert \leq R^2,\]
which concludes the proof.
\end{proof}

By a standard routine, we can now deduce (III) in all four cases:

\begin{cor}[Discrepancy estimates]\label{discr}
Let $\A = \mathbb{P}, \bS_2',\mathbb{L}',\mathbb{I'}$. 
Then for any $\alpha \in BAD$, we have (III) with $\varepsilon = \frac{1}{6} - \xi$ for any $\xi > 0$.
\end{cor}

\begin{proof}
We start with the case where $\A = \mathbb{P}$, which would already follow from the estimates obtained by Vaughan \cite{vaughan_primes}. Note that  
for $\mathcal{P} = \mathbb{P}$, we get
$\cS_N = \mathbb{P}\cap [\sqrt{N},N]$, and thus
$\#((\mathbb{P}\cap [0,N])\setminus \cS_N) = O(\sqrt{N})$.
Consequently, 
\[\begin{split}&\sup_{0 < a < b < 1} \left\vert\#\{n \in \mathbb{P} \cap [0,N]: \{n\alpha\} \in [a,b]\} - (b-a)\cdot \#\{n \in \mathbb{P} \cap [0,N]\}\right\rvert
\leq \\&\sup_{0 < a < b < 1} \left\vert\#\{n \in \cS_N: \{n\alpha\} \in [a,b]\} - (b-a)\cdot \#\{n \in \cS_N\}\right\rvert + O(\sqrt{N}).
\end{split}
\]
By the Erd\H{o}s--Turan inequality, we get for any $T \in \mathbb{N}$,
\[ \begin{split}\sup_{0 < a < b < 1} &\left\vert\#\{n \in \mathbb{P} \cap [0,N]: \{n\alpha\} \in [a,b]\} - (b-a) \cdot \#\{n \in \mathbb{P} \cap [0,N]\}\right\rvert
\\\ll &\sqrt{N} + \frac{N}{T} + \sum_{1\leq r \leq T}\frac{1}{r}\left\lvert \sum_{\substack{n \in \cS_N}}e(rn\alpha)\right\rvert.
\end{split}
\]
Since $\alpha \in BAD$, we find $a,q$ with $|\alpha - \frac{a}{q}| \leq \frac{1}{q^2}$ at any scale $q$.
Thus we may choose $q$ in Lemma \ref{sum_expsum} freely, giving for $q \asymp \sqrt{RN}$ and $R \leq N^{1/4}$ that
\begin{equation}\label{ram_exp}\sum_{r \leq R}\left\lvert \sum_{\substack{n \in \cS_N}}e(rn\alpha)\right\rvert \ll_{\xi}  N^{\xi}\left(
(RN)^{{3/4}} + N^{5/6} R^{2/3}\right).\end{equation}
Summation by parts and an application of \eqref{ram_exp} proves that for $T \leq N^{1/4}$, we get
\[ \sum_{1\leq r \leq T}\frac{1}{r}\left\lvert \sum_{\substack{n \in \cS_N}}e(rn\alpha)\right\rvert \ll_{\xi}  N^{5/6 + \xi}.
\]
We choose $T = N^{1/4}$, proving 

\[ \sup_{0 < a < b < 1} \left\vert\#\{p \leq N: \{p\alpha\} \in [a,b]\} - \#\{p \leq N\}\right\rvert
\ll_{\xi} N^{5/6 + \xi}.
\]

For $\A = \bS_2',\mathbb{L}',\mathbb{I}'$, we argue similarly. By making use of \eqref{sieve_s2} - \eqref{sieve_i}, we define
\[\mathcal{P} = \{p \equiv 3 \pmod 4\}, \quad \mathcal{P} = \{p \equiv 2 \pmod 3\}, \quad 
 \mathcal{P} = \{p \equiv 5,7,11 \pmod {12}\}\] respectively. This allows for another application of Lemma \ref{sum_expsum}, and we can conclude by the same arguments.
\end{proof}

\begin{proof}[Proof of Corollaries \ref{prime_cor} -- \ref{intersec_cor}]
    We start by proving Corollary \ref{prime_cor}. For this, we check that $\A = \mathbb{P}$ satisfies conditions (I),(II), and (III)
where $\p = \mathbb{P}$, $\rho = 1/2, \delta = 1$, and by Corollary \ref{discr}, $\varepsilon = \frac{1}{6}- \xi, \xi > 0$.
An application of Theorem \ref{thm_prime_like} thus shows
 \begin{equation*}
        BAD \subseteq \bigcap_{\psi \in \mathcal{M}_{\mathbb{P}}}\mathcal{K}_{\A}(\psi) \subseteq BAD_{\mathbb{P}}.
    \end{equation*}
    Since $\prod_{\substack{p \leq n\\p \in \mathbb{P}}}\left(1 + \frac{1}{p}\right) \asymp \log n$, Corollary \ref{prime_cor} follows.\\

    We next prove Corollary \ref{s2_cor}. For this, we check that $\bS_2'$ satisfies conditions (I),(II) and (III) with $\p = \{p \equiv 3 \pmod 4\}$.    
    Indeed, (I) follows from \eqref{sieve_s2} with $\rho = \tfrac{1}{2}$, (II) by Lemma \ref{lem_densities} with $\delta = \frac{1}{2}$
    and (III) by Corollary \ref{discr} with $\varepsilon = \frac{1}{6} - \xi$, $\xi > 0$.
    Thus applying Theorem \ref{thm_prime_like} proves
    $BAD \subseteq \bigcap_{\psi \in \mathcal{M}_{\bS_2'}}K_{\bS_2'}(\psi)$. By Lemma \ref{lem_densities}, the set $\bS_2'$ has positive relative density within $\bS_2$, thus Proposition \ref{pos_rel_dens_suff} implies
    \[BAD \subseteq \bigcap_{\psi \in \mathcal{M}_{\bS_2}}K_{\bS_2}(\psi).\]
    In order to show the converse for
    $\alpha$ satisfying $\liminf_{n \to \infty} n \lVert n \alpha\rVert \sqrt{\log n} = 0$, we use once more Lemma \ref{lem_densities}, and Proposition \ref{counterex_non_bad} with $f(x) = \sqrt{\log x}$.
    The proofs for Corollaries \ref{loesch_cor} and \ref{intersec_cor} work exactly in the same way, by only replacing $\p$ and $f$: For $\mathbb{L}'$, we use $\p = \{p \equiv 2 \pmod 3\}$ and $f(x) = \sqrt{\log x}$; for $\mathbb{I}'$, we use $\p = \{p \equiv 5,7,11 \pmod {12}\}$ and $f(x) = (\log x)^{3/4}$.
\end{proof}

\section{Proof of Theorem \ref{thm_pos_dens}}

Since $\A$ is a set of positive lower density, an application of Proposition \ref{counterex_non_bad} with $f(x) \equiv 1$, combined with the fact that $\mathcal{M}_{\A} = \mathcal{M}_{\N}$, proves that $\bigcap_{\psi \in \cM_{\A}}\mathcal{K}_{\A}(\psi) \subseteq BAD$, and it remains to show $BAD \subseteq \bigcap_{\psi \in \cM_{\A}}\mathcal{K}_{\A}(\psi)$.\\

To show the latter, let $\A$ be a set of positive lower density $\delta > 0$.
In view of Proposition \ref{wlog_psi}, we may assume \eqref{wlog_iii}. Further, we will assume that we have
\begin{equation}\label{dens_dyadic}d_k = \frac{\#\{2^{k} \leq n < 2^{k+1}: n \in \A\}}{2^k} \geq \frac{\delta}{2}\end{equation} for all $k$ sufficiently large. 
If this is not the case, we replace the basis $2$ by the basis $b := \max\{\lceil \frac{8}{\delta}\rceil,4\}$, and obtain that 
\[\frac{\#\{b^{k} \leq n < b^{k+1}: n \in \A\}}{b^{k+1}-b^{k}}
\geq \frac{\#\{n < b^{k+1}: n \in \A\}}{b^{k+1}} - \frac{1}{b-1} \geq \frac{\delta}{2} - \frac{2}{b} \geq \frac{\delta}{4}.
\]
By renaming $\delta' = \delta/4$ and replacing the basis $2$ in the definitions of $D_k,d_k,\mu_k,\psi_k$ with basis $b$, we can follow all estimates in exactly the same fashion. Thus for simplicity, we now assume for the remainder of this proof \eqref{dens_dyadic}. We aim to apply Lemma \ref{bdv_applied}, i.e., by making use of \eqref{wlog_pos}, the proof of Theorem \ref{thm_pos_dens} reduces to showing
\eqref{qia} and \eqref{equi}.\\

We start with proving \eqref{qia}:
By simply dropping the condition $i,j \in \A$, we have
\[\begin{split} \sum_{X \leq k,\ell \leq Y}\sum_{i \in D_k}\sum_{j \in D_{\ell}}\lambda(A_i \cap A_j)
&\leq 2\sum_{X \leq k \leq \ell \leq Y}\sum_{i \approx 2^k}\sum_{j \approx 2^{\ell}}\lambda(A_i \cap A_j)
\\&\ll \sum_{X \leq k \leq \ell \leq Y} \psi_{\ell} \sum_{i \approx 2^k}\#\{j \approx 2^{\ell}: \lVert (j -i)\alpha\rVert \leq 2\psi_k\},
\end{split}
\]
where we used the elementary estimate
\[\lambda(A_i \cap A_j) \ll \psi_{\ell}\mathds{1}_{[\lVert (j-i)\alpha\rVert \leq 2 \psi_k]}.\]
Since $\alpha \in BAD$, we have for any $t > 0$ that (recall \eqref{cardinality_bohr})

\[\#\{n \leq N: \lVert n \alpha \rVert \leq t\} \ll_{\alpha} Nt,\]
thus for $i \leq 2^{\ell}$, we get
\[\#\{j \approx 2^{\ell}: \lVert (j -i)\alpha\rVert \leq 2\psi_k\}
\ll_{\alpha} 2^{\ell}\psi_k.
\]
This implies 
\[\psi_{\ell} \sum_{i \approx 2^k}\#\{j \approx 2^{\ell}: \lVert (j -i)\alpha\rVert \leq 2\psi_k\}
\ll \psi_{\ell}\psi_k2^{\ell}2^{k} \ll_{\delta} \psi_{\ell}\psi_k\mu_{\ell}\mu_k,
\]
where we used \eqref{dens_dyadic}. This immediately shows \eqref{qia}, and we are left to prove \eqref{equi}.
We claim that this follows immediately from \eqref{dens_dyadic} and the assumption that $(n\alpha)_{n \in \A}$ is uniformly distributed: Indeed, let $\hat{\mu}_k := \#\{n < 2^k: n \in \A\}$. By assumption \eqref{dens_dyadic}, we get immediately
$\hat{\mu}_k \asymp_{\delta} \mu_k$.
We now observe that for any interval $I$, we have
\[\begin{split}\lambda(I) + o(1) &= \frac{\#\{n < 2^{k+1}: n \in \A, \{n\alpha\} \in I\}}{\mu_k + \hat{\mu}_k}
\\&=  \frac{\#\{n < 2^{k}: n \in \A, \{n\alpha\} \in I\}}{\mu_k+ \hat{\mu}_k} + \frac{\#\{2^k \leq n < 2^{k+1}: n \in \A, \{n\alpha\} \in I\}}{\mu_k+ \hat{\mu}_k}
\\&=  \frac{\lambda(I)\hat{\mu}_k + o(\hat{\mu}_k)}{\mu_k+ \hat{\mu}_k} + \frac{\#\{2^k \leq n < 2^{k+1}: n \in \A, \{n\alpha\} \in I\}}{\mu_k+ \hat{\mu}_k}, \quad k \to \infty.
\end{split}
\]
Rearranging and using $o(\hat{\mu}_k) = o_{\delta}(\mu_k)$, we get
\[\#\{2^k \leq n < 2^{k+1}: n \in \A, \{n\alpha\} \in I\} = \mu_k\lambda(I) + o_{\delta}(\mu_k),\] proving \eqref{equi}.
This proves Theorem \ref{thm_pos_dens} by an application of Lemma \ref{bdv_applied}.

\subsection{Proof of Corollary \ref{cor_pos_dens_theor}}
    We aim to show that 
    $(n\alpha)_{n \in \A_{a,f}}$ is uniformly distributed, which we show with Weyl's criterion, i.e., for every integer
    $h \geq 1$, we have for $\alpha \in BAD$ that
    \begin{equation}\label{exponential_sum}\sum_{\substack{n \leq N\\n \in \A_{a,f}}}e(hn\alpha) = o_h\left((\#[1,N] \cap \A_{a,f})\right) = o(N),\end{equation}
    where we used that $\mathcal{A}_{a,f}$ has positive lower density.
    Note that if $\alpha \in BAD$, then so is $h\alpha$, thus it suffices to show \eqref{exponential_sum} with $h = 1$.\\

Let us first consider the case of $a \neq 0$.
Since for $a \in \mu_m$ \[\sum_{j = 0}^{m-1} a^j = \begin{cases}
    m &\text{ if } a = 1,\\
    0 &\text{ otherwise},
\end{cases}\]
    we may write for $a \in \mu_m$
    \[\begin{split}\sum_{\substack{n \leq N\\n \in \A_{a,f}}}e(n\alpha)
    &=  \sum_{\substack{n \leq N}}e(n\alpha)\frac{1}{m}\sum_{j = 0}^{m-1}(f(n)\cdot a^{-1})^{j}
     \\&=  \frac{1}{m}\sum_{j = 0}^{m-1}a^{-j} \sum_{\substack{n \leq N}}e(n\alpha)f(n)^j.
     \end{split}
     \]  
    Note that $f_j(n) := f(n)^j$ is a multiplicative function satisfying $|f_j| \leq 1$.
    Thus we may appeal to a result of 
    Daboussi \cite{Daboussi} (see also the work of Montgomery--Vaughan \cite{mont_vau}) that in particular implies
    \[ \sum_{\substack{n \leq N}}e(n\alpha)f(n)^j = o(N), \quad 0 \leq j \leq m-1,\]
which immediately proves \eqref{exponential_sum}. This allows for an application of Theorem \ref{thm_pos_dens}, proving \eqref{full_kw}.
The case $a = 0$ is now immediate by
\[\sum_{\substack{n \leq N\\f(n) = 0}}e(n\alpha)
= \sum_{\substack{n \leq N}}e(n\alpha) - \sum_{a \in \mu_m}\sum_{\substack{n \leq N\\f(n) = a}}e(n\alpha) = o(N),
\]
by the preceding computation, and the trivial fact of $ \sum_{\substack{n \leq N}}e(n\alpha) = o(N)$.\\

We are left to prove that the situation applies to (a) - (d):
First, we prove that (a) - (d) are all of the form $\mathcal{A}_{a,f}$ for some $a,f$, which should be immediate to see: (a) follows from $f = \mu^2$ and $a = 1$, (b) and (c) from $f(n) := \xi_m^{\Omega(n)}$ respectively $f(n) = \xi_m^{\omega(n)}$, and (d) with $f(n) = \xi_m^n$ where $\xi_m = e(1/m)$. Thus in order to apply Theorem \ref{thm_pos_dens}, we need to show positive (lower) density for the respective set $\A_{a,f}$.
    It is a well-known fact that the set of square-free numbers has positive density, and the same is trivially true for (d). Further, it is probably standard for the sets in (b) and (c), but the author couldn't find a reference for this, so we provide a short proof.\\

    Using orthogonality relations as before, we have for $\hat{a} := e(a/m)$ that  \[\sum_{\substack{n \leq N\\\omega(n) \equiv a \pmod m}} 1 =  \frac{1}{m}\sum_{j = 0}^{m-1} \sum_{n \leq N} e\left(\frac{j(\omega(n) - a)}{m}\right)
    = \frac{N}{m} + O\left(\sum_{j = 1}^{m-1}\left\lvert  e\left(\frac{-j a}{m}\right)\right\rvert  \left\lvert \sum_{n \leq N} e\left(\frac{j\omega(n)}{m}\right)\right\rvert\right).
    \]
    Thus it suffices to show that for $j \not\equiv 0 \pmod m$,
    $ \sum_{n \leq N}  e\left(\frac{j\omega(n)}{m}\right) = o(N)$.
    Note that by the Prime Number Theorem, 
    \[\sum_{p \leq x}e\left(\frac{j\omega(p)}{m}\right)\log p = \xi_m^{j}\cdot x + O_A(x/(\log x)^A).\] 
We can then apply the Lagrange--Selberg--Delange method (see e.g. \cite[Theorem 13.2]{kouk_book}) to deduce
\[\sum_{n \leq N} e\left(\frac{j\omega(n)}{m}\right) \ll \frac{N}{(\log N)^{1 - Re(\xi_m^j)}}.\]
Since $\xi_m^j \neq 1$ for $j \not\equiv 0 \pmod m$, 
this shows in particular

\[\sum_{\substack{n \leq N\\\omega(n) \equiv a \pmod m}} 1 \sim \frac{N}{m},\quad  N \to \infty.
\]
Clearly, the same argument also holds for $\Omega(n)$ in place of $\omega(n)$.

\bibliographystyle{abbrv}
\bibliography{bibliography.bib}

\end{document}